\documentclass{article}
\usepackage[final]{graphicx}
\usepackage{psfrag}
\usepackage{amsmath,amsfonts,amssymb,amsxtra,subeqnarray,array}

\newlength{\extramargin}
\newcommand {\Real}{\ensuremath{{\mathbb{R}}}}

\newcommand {\Complex}{\ensuremath{{\mathbb{C}}}}

\newcommand{\R}{\ensuremath{\mathcal R}}
\newcommand{\Q}{\ensuremath{\mathcal Q}}
\newcommand{\D}{\ensuremath{\mathcal D}}

\newcommand{\setP}{\ensuremath{\mathcal P}}

\newcommand{\I}{\ensuremath{\mathcal I}}
\newcommand{\V}{\ensuremath{\mathcal V}}

\newcommand{\setS}{\ensuremath{\mathcal S}}

\newcommand{\M}{\ensuremath{\mathcal M}}

\newcommand{\setE}{\ensuremath{\mathcal E}}

\newcommand{\G}{\ensuremath{\mathcal G}}
\newcommand{\N}{\ensuremath{\mathcal N}}

\newcommand{\yu}{\ensuremath{{\mathbf{u}}}}

\newcommand{\ex}{\ensuremath{{\mathbf{x}}}}

\newcommand{\Abig}{\ensuremath{{\mathbf{A}}}}
\newcommand{\Bbig}{\ensuremath{{\mathbf{B}}}}
\newcommand{\Wbig}{\ensuremath{{\mathbf{W}}}}

\newcommand{\exr}{\ensuremath{{\mathbf{x}}_{\rm r}}}
\newcommand{\Abigr}{\ensuremath{{\mathbf{A}}_{\rm r}}}
\newcommand{\Bbigr}{\ensuremath{{\mathbf{B}}_{\rm r}}}
\newcommand{\Wbigr}{\ensuremath{{\mathbf{W}}_{\rm r}}}
\newcommand{\yur}{\ensuremath{{\mathbf{u}}_{\rm r}}}

\newtheorem{theorem}{Theorem}

\newtheorem{lemma}{Lemma}

\newtheorem{definition}{Definition}

\newtheorem{assumption}{Assumption}

\newtheorem{proposition}{Proposition}
\newenvironment{proof}{\noindent {\bf Proof.}}{\hfill \hspace*{1pt}\hfill$\blacksquare$}

\begin{document}
\title{Positive controllability of networks under relative actuation}
\author{S. Emre Tuna\footnote{The author is with Department of
Electrical and Electronics Engineering, Middle East Technical
University, 06800 Ankara, Turkey. The work has been completed
during his sabbatical stay at Department of Electrical and Electronic
Engineering, The University of Melbourne, Victoria 3010,
Australia. Email: {\tt etuna@metu.edu.tr}}} \maketitle

\begin{abstract}
For arrays of identical linear systems coupled through relative actuation four problems are studied: controllability, positive controllability, pairwise controllability, and positive pairwise controllability. To this end, related to the eigenvalues of the system matrix, certain graphs with possibly vector-valued edge weights are constructed. It is shown that array controllability and
graph connectivity are equivalent. Similar equivalences are established also between positive controllability \& strong connectivity, pairwise controllability \& pairwise connectivity, and positive pairwise controllability \& strong pairwise connectivity.
\end{abstract}

\section{Introduction}

Probably since Huygens pointed out the synchronization of two pendulum clocks, it must have been
self-evident that the collective behavior of a group of interacting systems should be determined by the connectivity of certain graph(s) representing (in some way) the interconnection between the individual units. What in general is not evident however is how to dig out the
graphs whose connectivity determines what need be determined.
\begin{figure}[h]
\begin{center}
\includegraphics[scale=0.4]{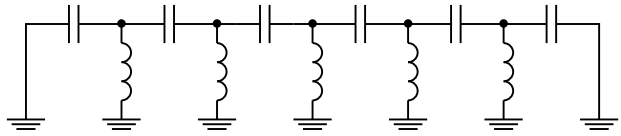}\\
\caption{10th order LC oscillator.}
\label{fig:oscillator}
\end{center}
\end{figure}
For instance, consider the individual electrical oscillator in Fig.~\ref{fig:oscillator}, where unit (1H) inductors are connected by unit (1F) capacitors as shown. Let us form two separate arrays, each containing three identical replicas of this oscillator coupled via unit (1$\Omega$) resistors as in Fig.~\ref{fig:arrays}a and Fig.~\ref{fig:arrays}b, respectively. Although neither array look more connected than the other to the eye, there is a significant  qualitative difference in their behaviors: starting from arbitrary initial conditions the oscillators in Fig.~\ref{fig:arrays}a always synchronize in the steady state, whereas those in Fig.~\ref{fig:arrays}b do not tend to oscillate in unison. This failure to synchronize can be traced back to the lack of connectivity of a certain graph.

\begin{figure}[h]
\begin{center}
\includegraphics[scale=0.3]{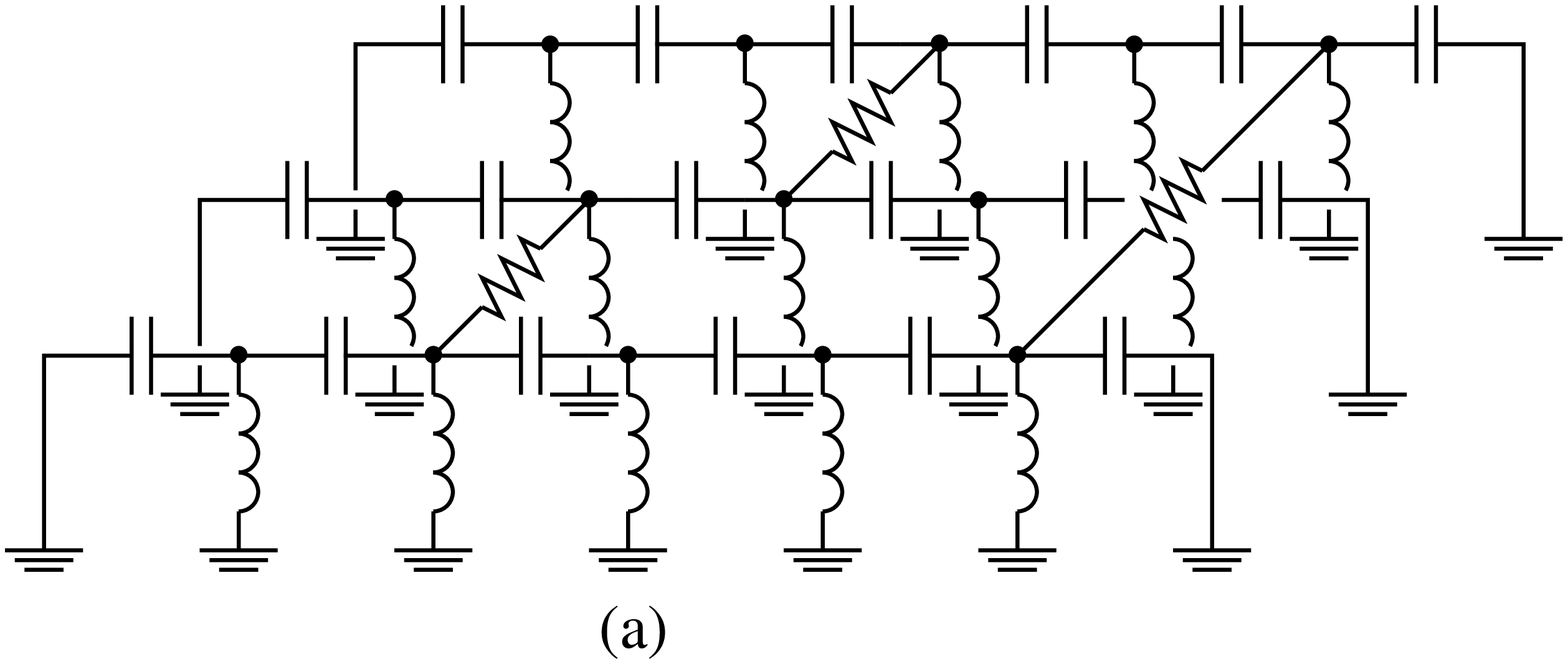}\includegraphics[scale=0.3]{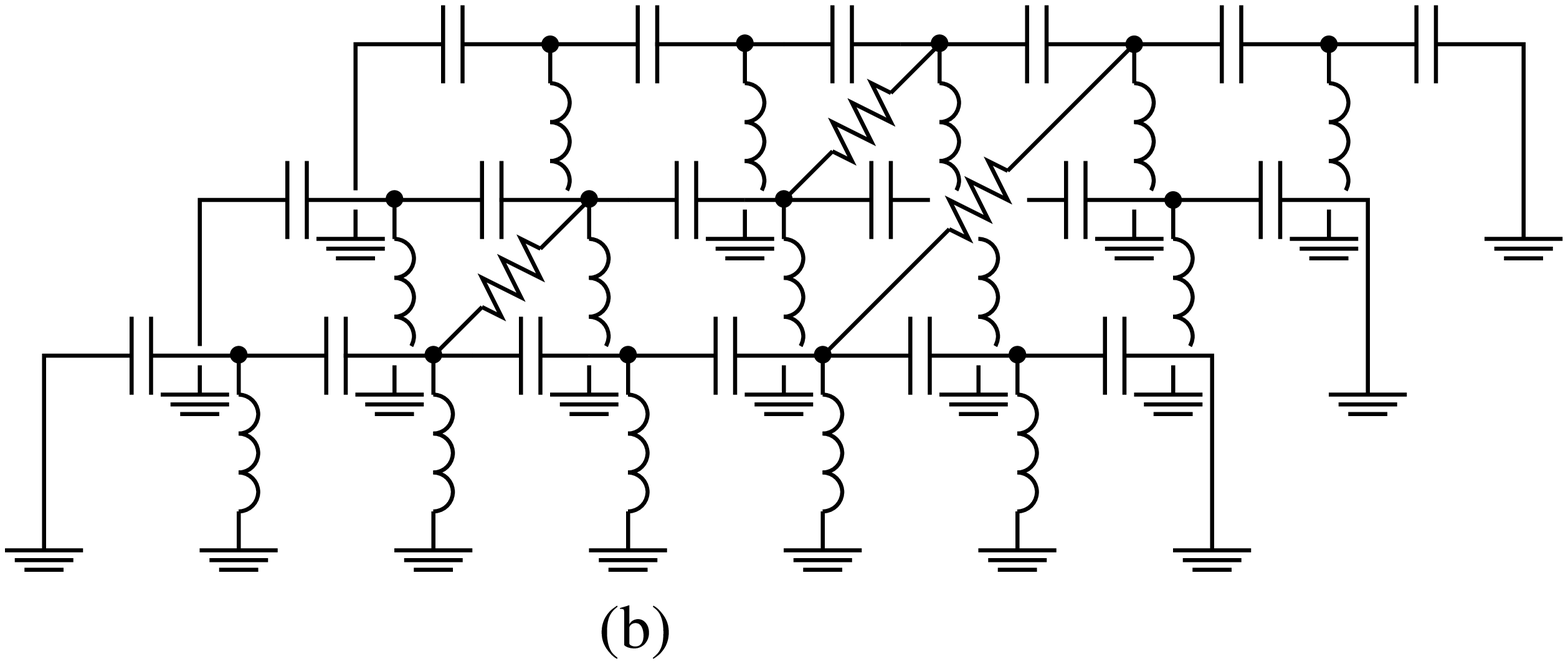}
\caption{Arrays of coupled oscillators. The array (a) synchronizes; the array (b) does not.}
\label{fig:arrays}
\end{center}
\end{figure}

Implicit in the above example is the importance of the role connectivity plays in network controllability. In fact, if the resistors in Fig.~\ref{fig:arrays}b are replaced by current sources (as our control inputs) the new array cannot be steered toward synchronization. The reason, not surprisingly, is the disconnectivity of the graph that was behind the failure of
synchronization in the old array. To see the relation between connectivity and controllability explicitly, let us visit a simpler example where the graph is not hidden. Consider three identical water tanks (integrators) connected via water pumps as shown in Fig.~\ref{fig:tanks}.
\begin{figure}[h]
\begin{center}
\includegraphics[scale=0.55]{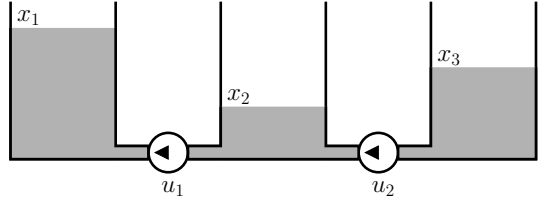}\\
\caption{Array of three water tanks.}
\label{fig:tanks}
\end{center}
\end{figure}
Letting $x_{i}$ denote the water volume (${\rm m}^{3}$) contained in the $i$th tank and $u_{\sigma}$ the flow rate (${\rm m}^{3}/{\rm h}$) through the $\sigma$th pump we can write the dynamics as
\begin{eqnarray}\label{eqn:watertanks}
\left[\begin{array}{c}{\dot x}_{1}\\
{\dot x}_{2}\\
{\dot x}_{3}\end{array}\right]=
\underbrace{\left[\begin{array}{rr}1&0\\
-1&1\\
0&-1\end{array}\right]}_{\displaystyle B}\left[\begin{array}{c}u_{1}\\
u_{2}\end{array}\right]\,.
\end{eqnarray}
The pleasant (zero column sum) structure of the matrix $B$ is shared by the incidence matrix  representing the graph $\Gamma$ in Fig.~\ref{fig:graph}. Observe that $\Gamma$ is (weakly) connected\footnote{In this paper by {\em connected} we mean {\em weakly connected}.} yet not strongly connected. This has two apparent implications. First, because the graph $\Gamma$ is connected the array is (relatively) controllable by which we mean that the relative states $x_{i}-x_{j}$ can be adjusted arbitrarily. That is, with bidirectional pumps the relative water levels can be simultaneously steered to any desired values regardless of the initial distribution. Second, because the graph $\Gamma$ is not strongly connected the array is not positively controllable. This translates to that with unidirectional pumps ($u_{\sigma}\geq 0$) the water levels cannot in general be equalized. At least three pumps are needed for that since at least three edges are needed for a 3-node graph to be strongly connected.

\begin{figure}[h]
\begin{center}
\includegraphics[scale=0.5]{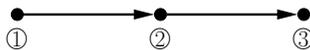}\\
\caption{The graph $\Gamma$ whose incidence matrix is $B$.}
\label{fig:graph}
\end{center}
\end{figure}

The water tanks example clearly illustrates the link between network controllability and graph connectivity. Meanwhile, as the oscillator array example indicates, the graphs whose connectivity determines controllability may not be apparent and therefore revealing them may require some effort. This paper is a report on such effort. Our setup is an array of linear time-invariant (LTI) systems driven by relative actuators. Specifically, the $i$th individual system's ($n$th order) dynamics reads ${\dot x}_{i}=Ax_{i}+\sum_{\sigma}B_{i\sigma}u_{\sigma}$ with $\sum_{i}B_{i\sigma}=0$. For this setup we study, from the connectivity point of view, four problems: controllability, positive controllability, pairwise controllability, and positive pairwise controllability.

{\em Controllability.} The literature on network controllability has so far concentrated on a somewhat different problem concerning a different setup than ours given above. The generally adopted node dynamics are first order and there is coupling between nodes even when the inputs are zero.
In addition, there is no relative actuation constraint. Namely, ${\dot x}_{i}=\sum_{j}a_{ij}x_{j}+\sum_{\sigma}b_{i\sigma}u_{\sigma}$ where $x_{i}\in\Real$.
Since the inputs are not relative, complete controllability is possible
and that is what has been thoroughly investigated. Generally speaking,
the problem that has received much attention
concerns with the question of how to achieve controllability
with as few inputs (or driver nodes) as possible; see, for instance, \cite{rahmani09,liu11,yuan13,olshevsky14}. In this particular direction a wealth of
results has accumulated, e.g., \cite{notarstefano13,commault13,pequito16,wang12,wang17}, starting possibly with Lin's work~\cite{lin74} on structural controllability. While these work dwell upon the ``how?'' for networks with first order node dynamics, we focus (in Section~\ref{sec:control}) on the simpler ``yes/no?'' for higher order dynamics. Namely, for an array with $q$ systems (nodes) and $p$ (relative) inputs, represented by the pair $[A,\,(B_{i\sigma})_{i,\sigma=1}^{q,p}]$ we present a necessary and sufficient condition for controllability\footnote{As mentioned earlier, when we use the word {\em controllable} to indicate an array we mean that all relative states $x_{i}-x_{j}$ (as opposed to actual states $x_{i}$) can be controlled. Since the actuation is relative in our setup, this is the most that can be achieved in terms of controllability. For instance, the total amount of water in the tanks in Fig.~\ref{fig:tanks} is constant and therefore independent of the control inputs driving the array.} from the graph connectivity point of view. The result is based on tools from classical control theory. The presented connectivity condition can indeed be seen as a certain reformulation of Popov-Belevitch-Hautus (PBH) test exploiting the special structure of our setup.

{\em Positive controllability.~\footnote{The term ``positive controllability'' seems to have been coined by Yoshida and his coauthors in the paper~\cite{yoshida94}.}} One of the earliest things we learn in life is how to steer a particular system (our body) with one-way actuators, for our muscles function that way. That is, a muscle can only pull or contract, but cannot (actively) push or extend. Another instance from biology of a one-way actuator is insulin, a key hormone in regulating the sugar level in blood. Insulin cannot undo what it does therefore pancreas employs another one-way agent, glucagon, to achieve proper regulation. Examples are not scarce outside biology; see, for instance, \cite{eden16} and references therein. The earliest work on controllability of LTI systems with positive controls (one-way actuators) is \cite{saperstone71}. Later Brammer provides a general eigenvector test~\cite{brammer72} which arguably is the most effective tool we have today on positive controllability of continuous-time LTI systems. Certain refinements/reformulations of Brammer's test are reported in \cite{heymann75,yoshida07}. Among the very few works on positive controllability of networks is \cite{lindmark16}, where the authors study first order node dynamics. Here, for arrays with $n$th order node dynamics, we provide in Section~\ref{sec:poscontrol} a necessary and sufficient strong connectivity condition for the positive controllability of an array. Just as our connectivity condition for controllability can be seen as a reformulation of PBH test, our strong connectivity condition for positive controllability is a natural extension of the refinements~\cite{yoshida07,lindmark16} of Brammer's eigenvector test.

{\em Pairwise controllability.} For an uncontrollable array, while it is not possible to steer all relative states, it is of interest to determine the subarrays of states that can be controlled relatively. The problem, at least for primitive arrays, is closely related to determining the connected components of an unconnected graph, which can be studied by means of paths connecting pairs of nodes. This motivates us to analyze (from connectivity point of view) the so called pairwise controllability, roughly described as follows. For a given pair $(k,\,\ell)$ of indices, an array is $(k,\,\ell)$-controllable when the difference $x_{k}-x_{\ell}$ can be arbitrarily adjusted. (The actual definition is subtler; see Def.~\ref{def:klcon}.) The outcome of our analysis is presented in Section~\ref{sec:paircon}, where we provide necessary and sufficient connectivity conditions for $(k,\,\ell)$-controllability. From the geometric point of view what is done is in effect checking whether a certain subspace (corresponding to $(k,\,\ell)$-controllability of the array) is contained in the overall controllable subspace.

{\em Positive pairwise controllability.} Last in our sequence of problems is the characterization of pairwise controllability of an array with positive controls. The off-the-shelf tools (such as controllability matrix, PBH test, Brammer's test) we use for the previous problems turn out not to be of much help here. Hence the analysis is of slightly different spirit and lengthier than before. However the end results (presented in Section~\ref{sec:pairposcon}) are of the same nature. In particular, positive pairwise controllability is interpreted in terms of strong connectivity of a pair of graph nodes.

To summarize, the contribution of this paper is intended to be in showing that the well-known close relation between controllability and connectivity for arrays with first order node dynamics naturally continue to exist for a class of arrays with higher order node dynamics. To this end, we study the above mentioned four facets of (relative) controllability. In particular, we establish connectivity characterizations of (pairwise) controllability as well as strong connectivity characterizations of positive (pairwise) controllability. With the possible exception of the contents of Section~\ref{sec:pairposcon}, the analysis methods employed in our work are not new; we borrow a great deal from the classical control theory toolbox. However, what we believe is fresh here is the perspective through which we tackle the problems at hand.

\section{Array}
A pair $[A,\,(B_{i\sigma})_{i,\sigma=1}^{q,p}]$ is meant to represent an {\em array} of $q\geq 2$ LTI systems
\begin{eqnarray}\label{eqn:array}
{\dot x}_{i}&=&Ax_{i}+\sum_{\sigma=1}^{p}B_{i\sigma}u_{\sigma}\,,\qquad i=1,\,2,\,\ldots,\,q
\end{eqnarray}
where $x_{i}\in\Real^{n}$ is the state of the $i$th system with $A\in\Real^{n\times n}$. The $\sigma$th (scalar) input is denoted by $u_{\sigma}\in\Real$. The input matrices
$B_{i\sigma}\in\Real^{n\times 1}$ are assumed to satisfy
\begin{eqnarray}\label{eqn:relative}
\sum_{i=1}^{q}B_{i\sigma}=0\,,\qquad \sigma=1,\,2,\,\ldots,\,p\,.
\end{eqnarray}
The constraint~\eqref{eqn:relative} means that the actuation is {\em relative}.
Hence the average of the states $x_{\rm av}=q^{-1}\sum x_{i}$ evolves
independently of the inputs driving the array, i.e., we have ${\dot x}_{\rm av}=Ax_{\rm av}$. The shorthand notation $(B_{::})$ represents the ordered collection $(B_{i\sigma})_{i,\sigma=1}^{q,p}$. Given some $d\times n$ matrix $M$, we write $(MB_{::})$ to mean the collection $(\tilde B_{i\sigma})_{i,\sigma=1}^{q,p}$ with $\tilde B_{i\sigma}=MB_{i\sigma}$.
For an index set $\I=\{\sigma_{1},\,\sigma_{2},\,\ldots,\,\sigma_{r}\}\subset\{1,\,2,\,\ldots,\,p\}$ the subcollection $(B_{i\sigma})_{i=1,\sigma\in\I}^{q}$ is denoted by $(B_{:\sigma})_{\sigma\in\I}$. The corresponding {\em incidence matrix} is constructed as
\begin{eqnarray*}
{\rm inc}\,(B_{:\sigma})_{\sigma\in\I}=
\left[\begin{array}{cccc}
B_{1\sigma_{1}}&B_{1\sigma_{2}}&\cdots&B_{1\sigma_{r}}\\
B_{2\sigma_{1}}&B_{2\sigma_{2}}&\cdots&B_{2\sigma_{r}}\\
\vdots&\vdots&\ddots&\vdots\\
B_{q\sigma_{1}}&B_{q\sigma_{2}}&\cdots&B_{q\sigma_{r}}
\end{array}\right]\,.
\end{eqnarray*}

\begin{definition}
An array $[A,\,(B_{::})]$ is said to be {\em controllable} if for each set of initial
conditions $(x_{1}(0),\,$ $x_{2}(0),\,\ldots,\,x_{q}(0))$  there exist a finite time $\tau>0$ and
input signals $u_{\sigma}:[0,\,\tau]\mapsto\Real$ such that $x_{i}(\tau)=x_{j}(\tau)$
for all $(i,\,j)$. The array is said to be {\em positively controllable} if the input signals
can be chosen to satisfy $u_{\sigma}:[0,\,\tau]\mapsto\Real_{\geq 0}$.
\end{definition}

\begin{definition}\label{def:klcon}
For a pair of distinct indices $k,\,\ell\in\{1,\,2,\,\ldots,\,q\}$
an array $[A,\,(B_{::})]$ is said to be {\em $(k,\,\ell)$-controllable} if for each set of initial conditions $(x_{1}(0),\,x_{2}(0),\,\ldots,\,x_{q}(0))$  there exist a finite time $\tau>0$ and input signals $u_{\sigma}:[0,\,\tau]\mapsto\Real$ such that $x_{k}(\tau)=x_{\ell}(\tau)$ and
$x_{i}(\tau)=e^{A\tau}x_{i}(0)$ for all $i\neq k,\,\ell$. The array is said to be {\em positively $(k,\,\ell)$-controllable} if the input signals can be chosen to satisfy $u_{\sigma}:[0,\,\tau]\mapsto\Real_{\geq 0}$.
\end{definition}

Our goal in this paper is to interpret the above definitions in terms of connectivity properties of certain graphs related to the array~\eqref{eqn:array}. In particular, we characterize (positive) controllability and (positive)  $(k,\,\ell)$-controllability in terms of (strong) connectivity and (strong) $(k,\,\ell)$-connectivity, respectively. Since our analysis heavily depends on graphs it is worthwhile to recall the relevant basics of graph theory. This we do next.

\section{Graph}

The next few definitions are borrowed from \cite{davis54}. The convex cone that is
{\em positively spanned} by the vectors $g_{1},\,g_{2},\,\ldots,\,g_{p}\in\Real^{q}$
is defined as
\begin{eqnarray*}
{\rm cone}\,\{g_{1},\,g_{2},\,\ldots,\,g_{p}\}=\left\{\zeta:\zeta=\sum_{\sigma = 1}^{p}\alpha_{\sigma}g_{\sigma},\,\alpha_{\sigma}\in\Real_{\geq 0}\right\}\,.
\end{eqnarray*}
In other words, ${\rm cone}\,\{g_{1},\,g_{2},\,\ldots,\,g_{p}\}$ is the set of all {\em positive combinations} of $g_{1},\,g_{2},\,\ldots,\,g_{p}$. For $\alpha\in\Real^{p}$ we write $\alpha\geq 0$
to mean $\alpha$ has no negative entry. Likewise, $\alpha\leq 0$ means $-\alpha\geq 0$. The convex cone spanned by the columns of a matrix $G$ is denoted by ${\rm cone}\,G$. That is, ${\rm cone}\,G=\{\zeta: \zeta=G\alpha,\,\alpha\geq 0\}$. The range and null spaces of $G$ are denoted by ${\rm range}\,G$ and ${\rm null}\,G$, respectively. The conjugate transpose of $G$ is denoted by $G^{*}$. (If $G$ is real then $G^{*}$ is simply the transpose of $G$.) The {\em synchronization subspace} is defined as $\setS_{n}={\rm range}\,[{\mathbf 1}_{q}\otimes I_{n}]$, where
${\mathbf 1}_{q}$ is the $q$-vector of all ones and
$I_{n}$ is the $n\times n$ identity matrix. $\setS_{n}^{\perp}$ denotes the orthogonal complement
of $\setS_{n}$. We say $G$ belongs to
{\em class-$\G_{n}$} ($G\in\G_{n}$) if ${\rm range}\,G\subset\setS_{n}^{\perp}$.
We let $e_{i}$ be the unit $q$-vector
with $i$th entry one and the remaining entries zero.

A (directed) graph $\Gamma=(\V,\,\setE,\,g)$ has a set of vertices (or nodes) $\V=\{v_{1},\,v_{2},\,\ldots,\,v_{q}\}$, a set of edges (or arcs)
$\setE=\{a_{1},\,a_{2},\,\ldots,\,a_{p}\}$, and a function $g:\setE\to\V\times\V$ that maps
each edge to an ordered pair $g(a_{\sigma})=(v_{i},\,v_{j})$ for some $i\neq j$. We allow parallel edges, i.e., $g$ need not be injective. By slight abuse of notation we sometimes call $(v_{i},\,v_{j})$ an edge and write
$(v_{i},\,v_{j})\in\setE$ when some $a_{\sigma}\in\setE$ exists satisfying $g(a_{\sigma})=(v_{i},\,v_{j})$.
Also, we write $-(v_{i},\,v_{j})$ to mean $(v_{j},\,v_{i})$.
A directed path from $v_{k}$ to $v_{\ell}$ ($k\neq \ell$) is a sequence of pairs $((v_{i_{1}},\,v_{i_{2}}),\,(v_{i_{2}},\,v_{i_{3}}),\,\ldots,\,(v_{i_{r}},\,v_{i_{r+1}}))$
satisfying $i_{1}=k$, $i_{r+1}=\ell$, and $(v_{i_{j}},\,v_{i_{j+1}})\in\setE$ for all
$j=1,\,2,\,\ldots,\,r$. An undirected path between $v_{k}$ and $v_{\ell}$ ($k\neq \ell$) is a sequence of pairs $((v_{i_{1}},\,v_{i_{2}}),\,(v_{i_{2}},\,v_{i_{3}}),\,\ldots,\,(v_{i_{r}},\,v_{i_{r+1}}))$
satisfying $i_{1}=k$, $i_{r+1}=\ell$, and for each $j=1,\,2,\,\ldots,\,r$ either $(v_{i_{j}},\,v_{i_{j+1}})$ or $(v_{i_{j+1}},\,v_{i_{j}})$ belongs to $\setE$. We adopt the convention that there is a (un)directed path from each vertex to itself despite that we allow no
loop edges $(v_{i},\,v_{i})$. For $k\neq\ell$ the graph $\Gamma$ is said to be {\em strongly $(k,\,\ell)$-connected} if there exist two directed paths, one from $v_{k}$ to $v_{\ell}$, the other from $v_{\ell}$ to $v_{k}$. It
is said to be {\em $(k,\,\ell)$-connected} if there exists an undirected path between
$v_{k}$ and $v_{\ell}$. It is said to be {\em (strongly) connected} if it is (strongly) $(k,\,\ell)$-connected for all $(k,\,\ell)$. The incidence matrix $[g_{i\sigma}]=G\in\Real^{q\times p}$ of the graph $\Gamma$ is such that the edge $a_{\sigma}$ with $g(a_{\sigma})=(v_{i},\,v_{j})$ is represented by the $\sigma$th column of $G$ in the following way: $g_{i\sigma}=1$, $g_{j\sigma}=-1$, and the remaining entries of the column are zero. I.e., $g(a_{\sigma})=(v_{i},\,v_{j})$ implies $\sigma$th column of $G$ equals $e_{i}-e_{j}$. Note that $G\in\G_{1}$ since ${\mathbf 1}_{q}^{*}G=0$. We now make the following simple observations.

\begin{proposition}\label{prop:graph}
Let $G\in\Real^{q\times p}$ be the incidence matrix of some graph $\Gamma=(\V,\,\setE,\,g)$. We have the following.
\begin{enumerate}
\item The graph $\Gamma$ is strongly $(k,\,\ell)$-connected if and only if ${\rm cone}\,G\supset{\rm range}\,[e_{k}-e_{\ell}]$.
    \item The graph $\Gamma$ is $(k,\,\ell)$-connected if and only if ${\rm range}\,G\supset{\rm range}\,[e_{k}-e_{\ell}]$.
\item The graph $\Gamma$ is strongly connected if and only if ${\rm cone}\,G\supset\setS_{1}^{\perp}$.
\item The graph $\Gamma$ is connected if and only if ${\rm range}\,G\supset\setS_{1}^{\perp}$.
\end{enumerate}
\end{proposition}

\begin{proof}
{\em 1.} If $\Gamma$ is strongly $(k,\,\ell)$-connected, a directed path $((v_{i_{1}},\,v_{i_{2}}),\,(v_{i_{2}},\,v_{i_{3}}),\,\ldots,\,(v_{i_{r}},\,v_{i_{r+1}}))$ exists
with $i_{1}=k$, $i_{r+1}=\ell$, and $(v_{i_{j}},\,v_{i_{j+1}})\in\setE$ for all
$j=1,\,2,\,\ldots,\,r$. This implies each $e_{i_{j}}-e_{i_{j+1}}$ is a column of $G$ for all
$j=1,\,2,\,\ldots,\,r$. Since we can write
\begin{eqnarray*}
(e_{i_{1}}-e_{i_{2}})+(e_{i_{2}}-e_{i_{3}})+\cdots+(e_{i_{m}}-e_{i_{m+1}})=e_{k}-e_{\ell}
\end{eqnarray*}
we have $e_{k}-e_{\ell}\in{\rm cone}\,G$. Likewise, the existence of a directed path from $v_{\ell}$ to $v_{k}$ yields $e_{\ell}-e_{k}\in{\rm cone}\,G$. Therefore ${\rm cone}\,G\supset\{e_{k}-e_{\ell},\,e_{\ell}-e_{k}\}$, which implies ${\rm cone}\,G\supset{\rm range}\,[e_{k}-e_{\ell}]$.

Let us establish the other direction by contradiction. Suppose there does not exist a directed path from $v_{k}$ to $v_{\ell}$ while ${\rm cone}\,G\supset{\rm range}\,[e_{k}-e_{\ell}]$. Let $\V_{k}\subset\V$ be the set of all vertices to which there is a directed path from $v_{k}$. Similarly, let $\V_{\ell}\subset\V$ be the set of all vertices from which there is a directed path to $v_{\ell}$. (Note that $v_{k}\in\V_{k}$ and $v_{\ell}\in\V_{\ell}$.) Since no directed path exists from $v_{k}$ to $v_{\ell}$, the sets $\V_{k}$ and $\V_{\ell}$ are disjoint. Now define the edge sets $\setE_{k}=\{(v_{i},\,v_{j})\in\setE:v_{i},\,v_{j}\in\V_{k}\}$ and $\setE_{\ell}=\{(v_{i},\,v_{j})\in\setE:v_{i},\,v_{j}\in\V_{\ell}\}$, which, too, are disjoint.
Let $G_{k}$ and $G_{\ell}$ be the incidence matrices of the subgraphs $\Gamma_{k}=(\V_{k},\,\setE_{k},\,g)$ and $\Gamma_{\ell}=(\V_{\ell},\,\setE_{\ell},\,g)$, respectively. Since the vertices and edges can be relabeled, we can assume, without loss of generality, $G$ has the following block structure.
\begin{eqnarray*}
G=\left[\begin{array}{c}\!\!\!\!
\begin{array}{c}
G_{k}\\ 0\\ 0
\end{array}
\vline
\begin{array}{l}
M_{-}\\ M_{*}\\ M_{+}
\end{array}
\vline\begin{array}{c}
0 \\ 0\\ G_{\ell}
\end{array}\!\!\!\!
\end{array}\right]=
\left[\begin{array}{clc}
G_{k}\!\!\!\! & M_{-}\!\!\!\! & 0\\
\hline
\vspace{-0.13in}\\
0\!\!\!\! & M_{*}\!\!\!\! & 0\\
\hline
\vspace{-0.13in}\\
0\!\!\!\! & M_{+}\!\!\!\! & G_{\ell}
\end{array}\right]
\end{eqnarray*}

For the ease of discussion $G$ is partitioned in various ways as shown above. Observe that the entries of the matrix $M_{-}$ (if exists) are either $0$ or $-1$. To see that suppose
otherwise, i.e., $M_{-}$ has an entry $g_{i\sigma}=1$. Let $j$ be such that $g_{j\sigma}=-1$.
Since $g_{i\sigma}$ belongs to the center column partition, the edge $a_{\sigma}$ (satisfying $g(a_{\sigma})=(v_{i},\,v_{j})$) can belong neither to $\setE_{k}$ nor to $\setE_{\ell}$. Moreover, $v_{i}\in\V_{k}$ because $g_{i\sigma}$ belongs to the upper row partition. By definition, $v_{i}\in\V_{k}$ implies there is a directed path from $v_{k}$ to $v_{i}$. The existence of the edge $(v_{i},\,v_{j})$ implies that this path can be extended to $v_{j}$. Consequently $v_{j}\in \V_{k}$. Since both vertices $v_{i},\,v_{j}$ belong to $\V_{k}$ we have to have $a_{\sigma}\in\setE_{k}$, but this we already ruled out. Therefore $M_{-}$ can have only nonpositive entries.

Since ${\rm cone}\,G\supset{\rm range}\,[e_{k}-e_{\ell}]$ we can find a $p$-vector $\alpha\geq 0$ satisfying $G\alpha=e_{k}-e_{\ell}$. Let $q_{k}$ be the number of vertices in $\V_{k}$ and $q_{\ell}$ the number of vertices in $\V_{\ell}$. We can write
\begin{eqnarray*}
G\alpha=\left[\begin{array}{clc}
G_{k}\!\!\!\! & M_{-}\!\!\!\! & 0\\
\hline
\vspace{-0.13in}\\
0\!\!\!\! & M_{*}\!\!\!\! & 0\\
\hline
\vspace{-0.13in}\\
0\!\!\!\! & M_{+}\!\!\!\! & G_{\ell}
\end{array}\right]\alpha=\left[\begin{array}{c}
f_{+}\\
\hline
\vspace{-0.13in}\\
f_{0}\\
\hline
\vspace{-0.13in}\\
f_{-}
\end{array}\right]=e_{k}-e_{\ell}
\end{eqnarray*}
where $f_{+}\in\Real^{q_{k}}$ and $f_{-}\in\Real^{q_{\ell}}$. Since $k\leq q_{k}$ and $\ell\geq q-q_{\ell}+1>q_{k}$ the vector $f_{+}=[G_{k}\ M_{-}\ 0]\alpha$ contains exactly one $+1$ entry while its other entries are zero. (Likewise, $f_{-}$ contains exactly one $-1$ entry while its other entries are zero and $f_{0}$ (if exists) is a vector of zeros.) In particular, ${\bf 1}_{q_{k}}^{*}f_{+}=1$. Now define
$\beta=[G_{k}\ M_{-}\ 0]^{*}{\bf 1}_{q_{k}}$. Using $G_{k}^{*}{\bf 1}_{q_{k}}=0$ (because $G_{k}$ is an incidence matrix) and the fact that $M_{-}$ has no positive entry we can write
$\beta
=[G_{k}\ M_{-}\ 0]^{*}{\bf 1}_{q_{k}}
=[0\ M_{-}\ 0]^{*}{\bf 1}_{q_{k}}
\leq 0$. Combining $\beta\leq 0$ with $\alpha\geq 0$ yields $\beta^{*}\alpha\leq 0$.
However this results in the following contradiction
\begin{eqnarray*}
1={\bf 1}_{q_{k}}^{*}f_{+}={\bf 1}_{q_{k}}^{*}[G_{k}\ M_{-}\ 0]\alpha=\beta^{*}\alpha\leq 0\,.
\end{eqnarray*}

{\em 2.} Given $\Gamma=(\V,\,\setE,\,g)$ with $\V=\{v_{1},\,v_{2},\,\ldots,\,v_{q}\}$ and $\setE=\{a_{1},\,a_{2},\,\ldots,\,a_{p}\}$, define the mapping $g_{\rm a}:\setE_{\rm a}\to\V\times\V$ for the {\em augmented} edge set $\setE_{\rm a}=\{a_{1},\,a_{2},\,\ldots,\,a_{2p}\}$ as follows.
\begin{eqnarray*}
g_{\rm a}(a_{\sigma})=\left\{\begin{array}{rcl}
g(a_{\sigma})&\mbox{for}&\sigma=1,\,2,\,\ldots,\,p\\
-g(a_{\sigma})&\mbox{for}&\sigma=p+1,\,p+2,\,\ldots,\,2p\,.
\end{array}
\right.
\end{eqnarray*}
Let $\Gamma_{\rm a}=(\V,\,\setE_{\rm a},\,g_{\rm a})$.
It is not hard to see that $\Gamma$ is $(k,\,\ell)$-connected if and only if
$\Gamma_{\rm a}$ is strongly $(k,\,\ell)$-connected. Moreover, the incidence matrix of $\Gamma_{\rm a}$ reads $G_{\rm a}=[G\ -G]$. Therefore ${\rm cone}\,G_{\rm a}={\rm range}\,G$. Using the first statement of the proposition we can now write
\begin{eqnarray*}
\mbox{$\Gamma$ $(k,\,\ell)$-connected}
&\Longleftrightarrow&\mbox{$\Gamma_{\rm a}$ strongly $(k,\,\ell)$-connected}\\
&\Longleftrightarrow&{\rm cone}\,G_{\rm a}\supset{\rm range}\,[e_{k}-e_{\ell}]\\
&\Longleftrightarrow&{\rm range}\,G\supset{\rm range}\,[e_{k}-e_{\ell}]\,.
\end{eqnarray*}

{\em 3.} If $\Gamma$ is strongly connected, then, by definition, $\Gamma$ is strongly $(k,\,\ell)$-connected for all pairs $(k,\,\ell)$. The first statement of the proposition then allows us
to write
\begin{eqnarray*}
{\rm cone}\,G
\supset\sum_{k,\ell}\,{\rm range}\,[e_{k}-e_{\ell}]
=\setS_{1}^{\perp}\,.
\end{eqnarray*}
If $\Gamma$ is not strongly connected, then there exists a pair $(k,\,\ell)$ for which ${\rm cone}\,G\not\supset{\rm range}\,[e_{k}-e_{\ell}]$. Since $\setS_{1}^{\perp}\supset{\rm range}\,[e_{k}-e_{\ell}]$ we have to have ${\rm cone}\,G
\not\supset\setS_{1}^{\perp}$.

{\em 4.} The result follows from the second statement. The demonstration is similar to that of the third statement.
\end{proof}

\vspace{0.12in}

Proposition~\ref{prop:graph} motivates us for the following generalization.

\begin{definition}\label{def:graph}
A class-$\G_{n}$ matrix $G\in\Complex^{(qn)\times p}$ is said to be:
\begin{itemize}
\item {\em strongly $(k,\,\ell)$-connected} if ($G$ is real and) ${\rm cone}\,G\supset{\rm range}\,[(e_{k}-e_{\ell})\otimes I_{n}]$\,,
    \item {\em $(k,\,\ell)$-connected} if ${\rm range}\,G\supset{\rm range}\,[(e_{k}-e_{\ell})\otimes I_{n}]$\,,
\item {\em strongly connected} if ($G$ is real and) ${\rm cone}\,G\supset\setS_{n}^{\perp}$\,,
\item {\em connected} if ${\rm range}\,G\supset\setS_{n}^{\perp}$\,.
\end{itemize}
\end{definition}

A brief digression is in order here. Definition~\ref{def:graph} intends to extend connectivity, a central notion for graphs, to class-$\G_{n}$ matrices, which may be taken to represent (or be) generalized graphs. In this representation each column of $G\in\G_{n}$ may be treated as an edge. Then a path between $k$th and $\ell$th vertices may be said to exist if for each $\eta\in\Real^{n}$ we can find some edges (columns) $g_{\sigma}$ and some weights $\alpha_{\sigma}$
that take us from vertex $k$ to vertex $\ell$ by satisfying $e_{k}\otimes\eta+\alpha_{1}g_{1}+\alpha_{2}g_{2}+\cdots
+\alpha_{r}g_{r}=e_{\ell}\otimes\eta$. (For a directed path we would require the weights to be positive.) Depicting the endpoints $e_{k}\otimes\eta$ and $e_{\ell}\otimes\eta$ as dots (in space they belong to) and the vectors $\alpha_{\sigma}g_{\sigma}$ as successive line segments connecting the two dots, a geometric interpretation can be obtained. Hence, although it would be difficult (provided it is possible/meaningful) to draw a generalized graph, the notion of connectivity seems to maintain to some degree its visual feature.

Definition~\ref{def:graph} has an interesting claim on hypergraphs,\footnote{A hypergraph is such that each edge is allowed to be incident to more than two vertices.} which can be observed on a simple instance. Consider the array~\eqref{eqn:watertanks} of water tanks under the constraint
$u_{2}=2u_{1}$ which arises, say, because the voltages driving the water pumps are not independent. Then the dynamics reads $[{\dot x}_{1}\ {\dot x}_{2}\ {\dot x}_{3}]^{*}=[1\ \ 1\ -2]^{*}u_{1}$. Let now $G=[1\ \ 1\ -2]^{*}\in\G_{1}$ represent the incidence matrix of a 3-vertex graph; call this graph $\Gamma$. Since $G$ has a single column, $\Gamma$ has a single edge. This edge is incident to all the vertices, for the corresponding column has no zero entries. Therefore $\Gamma$ is a 3-vertex hypergraph  with a single (hyper)edge that is incident to all three vertices. According to the classical definition of connectivity for hypergraphs, $\Gamma$ is connected because any two vertices are adjacent to one another through the only edge. According to Definition~\ref{def:graph} however, $G$ (therefore $\Gamma$) is not connected. In fact, no two vertices are connected since ${\rm range}\,G\not\supset{\rm range}\,[e_{k}-e_{\ell}]$ for all pairs $(k,\,\ell)$. To support Definition~\ref{def:graph} against this discrepancy let us first obtain the Laplacian matrix $L$
of $\Gamma$ from the incidence matrix $G$ as
\begin{eqnarray*}
L=GG^{*}=\left[\begin{array}{rrr}1&1&-2\\1&1&-2\\-2&-2&4\end{array}\right]\,.
\end{eqnarray*}
Then, treating $L$ as the node admittance matrix of a resistive network, we obtain
the conductances $\gamma_{ij}$ between the nodes $i$ and $j$ through
\begin{eqnarray*}
\left[\begin{array}{ccc}
\gamma_{12}+\gamma_{31}&-\gamma_{12}&-\gamma_{31}\\
-\gamma_{12}&\gamma_{23}+\gamma_{12}&-\gamma_{23}\\
-\gamma_{31}&-\gamma_{23}&\gamma_{31}+\gamma_{23}
\end{array}\right]=L
\end{eqnarray*}
as $\gamma_{12}=-1\mho$, $\gamma_{23}=2\mho$, and $\gamma_{31}=2\mho$. These values yield the simple delta network in Fig.~\ref{fig:delta}.
\begin{figure}[h]
\begin{center}
\includegraphics[scale=0.5]{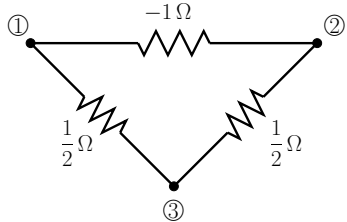}\\
\caption{The delta network with node admittance matrix $L$.}
\label{fig:delta}
\end{center}
\end{figure}
Now, the connectivity of this network can be determined via the circuit theory tool {\em effective conductance}. Namely, the vertices $k$ and $\ell$ of the graph $\Gamma$ is connected if the effective conductance between the nodes $k$ and $\ell$ of the corresponding resistive network in Fig.~\ref{fig:delta} is positive. A quick calculation shows that the effective conductance for any pair of nodes is zero for this network. Hence no two vertices of $\Gamma$ is connected, just as Definition~\ref{def:graph} predicates.

In the remainder of the paper we attempt to interpret different controllability aspects of the array~\eqref{eqn:array} in terms of (strong) connectivity properties of certain class-$\G_{n}$ matrices.

\section{Controllability}\label{sec:control}

Consider the array~\eqref{eqn:array}. By $\mu_{1},\,\mu_{2},\,\ldots,\,\mu_{m}$ we denote the distinct eigenvalues of $A^{*}$. Note that $m\leq n$ and these eigenvalues are shared by $A$. For each $\kappa\in\{1,\,2,\,\ldots,\,m\}$ we let $V_{\kappa}\in\Complex^{n\times d_{\kappa}}$ denote a full column rank matrix satisfying ${\rm range}\,V_{\kappa}={\rm
null}\,[A^{*}-\mu_{\kappa}I_{n}]$. Therefore $d_{\kappa}$ is the geometric multiplicity of the
eigenvalue $\mu_{\kappa}$. We let $V_{\kappa}\in\Real^{n\times d_{\kappa}}$ when $\mu_{\kappa}\in\Real$.
Note that the columns of $V_{\kappa}$ are the
linearly independent eigenvectors of $A^{*}$ corresponding to the
eigenvalue $\mu_{\kappa}$. In particular, we have
$A^{*}V_{\kappa}=\mu_{\kappa}V_{\kappa}$. For notational convenience we sometimes represent the array~\eqref{eqn:array} as a single big system
\begin{eqnarray}\label{eqn:bigsystem}
{\dot \ex}=\Abig\ex+\Bbig\yu
\end{eqnarray}
where $\ex=[x_{1}^{*}\ x_{2}^{*}\ \cdots\ x_{q}^{*}]^{*}$ is the state and
$\yu=[u_{1}\ u_{2}\ \cdots\ u_{p}]^{*}$ is the input. Clearly, we have
\begin{eqnarray*}
\Abig=[I_{q}\otimes A]
\end{eqnarray*}
while $\Bbig\in\Real^{(qn)\times p}$ has the structure
\begin{eqnarray*}
\Bbig={\rm inc}\,(B_{::})=
\left[\begin{array}{cccc}
B_{11}&B_{12}&\cdots&B_{1p}\\
B_{21}&B_{22}&\cdots&B_{2p}\\
\vdots&\vdots&\ddots&\vdots\\
B_{q1}&B_{q2}&\cdots&B_{qp}
\end{array}\right]\,.
\end{eqnarray*}
Note that $\Bbig\in\G_{n}$ due to \eqref{eqn:relative} and ${\rm inc}\,(V_{\kappa}^{*}B_{::})=[I_{q}\otimes V_{\kappa}^{*}]\Bbig$.
The controllability matrix associated to the pair $[\Abig,\,\Bbig]$
reads $[\Bbig\ \Abig\Bbig\ \cdots\ \Abig^{qn-1}\Bbig]$. However, being equal to $[I_{q}\otimes A]$, the matrix $\Abig$ satisfies the characteristic equation of $A$ (which is of order $n$) and therefore the controllability index for the pair $[\Abig,\,\Bbig]$ is at most $n$. Hence we can treat
\begin{eqnarray*}
\Wbig=[\Bbig\ \Abig\Bbig\ \cdots\ \Abig^{n-1}\Bbig]
\end{eqnarray*}
as the controllability matrix.
Indeed, ${\rm range}\,\Wbig$ is the controllable subspace for the pair $[\Abig,\,\Bbig]$.
Observe $\Wbig\in\G_{n}$. Recall that the vector ${\bf 1}_{q}$ spans the synchronization subspace $\setS_{1}$. Let $S$ denote its normalized version, i.e., $S={\bf 1}_{q}/\sqrt{q}$ and hence
$S^{*}S=1$. Also, let $D\in\Real^{q\times (q-1)}$
be some matrix whose columns make an orthonormal basis for $\setS_{1}^{\perp}$, sometimes called the {\em disagreement subspace}. Note that $D^{*}D=I_{q-1}$ and the columns of the matrix $[D\ S]$ make an orthonormal basis for $\Real^{n}$. The following identities are easy to show and find frequent use in the sequel.
\begin{enumerate}
\item[(i)]$DD^{*}+SS^{*}=I_{q}$.
\item[(ii)] ${\rm range}\,[D\otimes I_{n}]=\setS_{n}^{\perp}$.
\item[(iii)] ${\rm null}\,[D^{*}\otimes I_{n}]=\setS_{n}$.
\item[(iv)] $[S^{*}\otimes I_{n}]\Bbig=0$.
\end{enumerate}
The distance of $\ex$ to $\setS_{n}$ we denote by $\|\ex\|_{\setS_{n}}=\left\|[D^{*}\otimes I_{n}]\ex\right\|$. Finally, we define the {\em reduced} parameters
\begin{eqnarray*}
\Abigr&=&[I_{q-1}\otimes A]\\
\Bbigr&=&[D^{*}\otimes I_{n}]\Bbig\\
\Wbigr&=&[\Bbigr\ \Abigr\Bbigr\ \cdots\ \Abigr^{n-1}\Bbigr]\,.
\end{eqnarray*}
The controllability index for the pair $[\Abigr,\,\Bbigr]$ is at most $n$. Hence ${\rm range}\,\Wbigr$ equals the controllable subspace associated to the pair $[\Abigr,\,\Bbigr]$.

\begin{lemma}\label{lem:con}
The following are equivalent.
\begin{enumerate}
\item The array $[A,\,(B_{::})]$ is controllable.
\item ${\rm range}\,\Wbig\supset\setS_{n}^{\perp}$.
\item The pair $[\Abigr,\,\Bbigr]$ is controllable.
\end{enumerate}
\end{lemma}

\begin{proof}
{\em 1$\implies$3.} Suppose $[A,\,(B_{::})]$ is controllable. For the system~\eqref{eqn:bigsystem} this means every initial condition $\ex(0)$ can be driven to
$\setS_{n}$ in finite time. Consider the system ${\dot \ex}_{\rm r}=\Abigr\exr+\Bbigr\yur$.
Choose an arbitrary initial condition $\exr(0)\in(\Real^n)^{q-1}$. Now set the initial condition
of the system~\eqref{eqn:bigsystem} as $\ex(0)=[D\otimes I_{n}]\exr(0)$. Then let $\yu:[0,\,\tau]\to\Real^{p}$ be
an input signal that yields $\|\ex(\tau)\|_{\setS_{n}}=0$ for some finite $\tau>0$. Such input signal exists thanks to the controllability of the pair $[A,\,(B_{::})]$.
Let $\yur(t)=\yu(t)$ for $t\in[0,\,\tau]$. Recall that $\|\ex(\tau)\|_{\setS_{n}}=0$ means $[D^{*}\otimes I_{n}]\ex(\tau)=0$.
We can write
\begin{eqnarray*}
\exr(\tau)
&=&e^{\Abigr\tau}\exr(0)+\int_{0}^{\tau}e^{\Abigr(\tau-t)}\Bbigr\yur(t)dt\\
&=&[I_{q-1}\otimes e^{A\tau}]\exr(0)+\int_{0}^{\tau}[I_{q-1}\otimes e^{A(\tau-t)}]\Bbigr\yur(t)dt\\
&=&[D^{*}D\otimes e^{A\tau}]\exr(0)+\int_{0}^{\tau}[I_{q-1}\otimes e^{A(\tau-t)}][D^{*}\otimes I_{n}]\Bbig\yur(t)dt\\
&=&[D^{*}\otimes I_{n}][I_{q}\otimes e^{A\tau}][D\otimes I_{n}]\exr(0)+[D^{*}\otimes I_{n}]\int_{0}^{\tau}[I_{q}\otimes e^{A(\tau-t)}]\Bbig\yur(t)dt\\
&=&[D^{*}\otimes I_{n}][I_{q}\otimes e^{A\tau}]\ex(0)+[D^{*}\otimes I_{n}]\int_{0}^{\tau}[I_{q}\otimes e^{A(\tau-t)}]\Bbig\yu(t)dt\\
&=&[D^{*}\otimes I_{n}]\left(e^{\Abig\tau}\ex(0)+\int_{0}^{\tau}e^{\Abig(\tau-t)}\Bbig\yu(t)dt\right)\\
&=&[D^{*}\otimes I_{n}]\ex(\tau)\\
&=&0\,.
\end{eqnarray*}
Hence the pair $[\Abigr,\,\Bbigr]$ must be controllable because $\exr(0)$ was arbitrary.

{\em 3$\implies$2.} Suppose $[\Abigr,\,\Bbigr]$ is controllable.
Observe that we can write for all $k\in\{0,\,1,\,\ldots,\,n-1\}$
\begin{eqnarray*}
\Abigr^{k}\Bbigr
&=&[I_{q-1}\otimes A]^{k}[D^{*}\otimes I_{n}]\Bbig\\
&=&[I_{q-1}\otimes A^{k}][D^{*}\otimes I_{n}]\Bbig\\
&=&[D^{*}\otimes I_{n}][I_{q}\otimes A^{k}]\Bbig\\
&=&[D^{*}\otimes I_{n}][I_{q}\otimes A]^{k}\Bbig\\
&=&[D^{*}\otimes I_{n}]\Abig^{k}\Bbig
\end{eqnarray*}
yielding $\Wbigr=[D^{*}\otimes I_{n}]\Wbig$. Similarly, using $[S^{*}\otimes I_{n}]\Bbig=0$ we
can write
\begin{eqnarray*}
[S^{*}\otimes I_{n}]\Abig^{k}\Bbig
&=&[S^{*}\otimes I_{n}][I_{q}\otimes A^{k}]\Bbig\\
&=&A^{k}[S^{*}\otimes I_{n}]\Bbig\\
&=&0
\end{eqnarray*}
yielding $[S^{*}\otimes I_{n}]\Wbig=0$. Since $[\Abigr,\,\Bbigr]$ is controllable, $\Wbigr$ is
full column rank, which allows us to write ${\rm range}\,[D\otimes I_{n}]\Wbigr={\rm range}\,[D\otimes I_{n}]$. Let us now gather our recent findings and obtain
\begin{eqnarray*}
{\rm range}\,\Wbig
&=&{\rm range}\,[(DD^{*}+SS^{*})\otimes I_{n}]\Wbig\\
&=&{\rm range}\,\left([D\otimes I_{n}][D^{*}\otimes I_{n}]\Wbig+[S\otimes I_{n}][S^{*}\otimes I_{n}]\Wbig\right)\\
&=&{\rm range}\,[D\otimes I_{n}]\Wbigr\\
&=&{\rm range}\,[D\otimes I_{n}]\\
&=&\setS_{n}^{\perp}\,.
\end{eqnarray*}
Therefore ${\rm range}\,\Wbig\supset\setS_{n}^{\perp}$.

{\em 2$\implies$1.} Suppose ${\rm range}\,\Wbig\supset\setS_{n}^{\perp}$. Consider the system~\eqref{eqn:bigsystem}. Let $\ex(0)$ be an arbitrary initial condition. Then let
$\xi=[DD^{*}\otimes I_{n}]\ex(0)$. Note that $\xi\in\setS_{n}^{\perp}$. Therefore $\xi$ belongs to
the controllable subspace ${\rm range}\,\Wbig$ associated to the pair $[\Abig,\,\Bbig]$.
Consequently we can find an input signal $\yu:[0,\,\tau]\to\Real^{p}$ with finite $\tau>0$
that satisfies
\begin{eqnarray*}
e^{\Abig\tau}\xi+\int_{0}^{\tau}e^{\Abig(\tau-t)}\Bbig\yu(t)dt=0\,.
\end{eqnarray*}
Now using this control signal and the identity $D^{*}S=0$ we can write
\begin{eqnarray*}
[D^{*}\otimes I_{n}]\ex(\tau)
&=&[D^{*}\otimes I_{n}]\left(
e^{\Abig\tau}\ex(0)+\int_{0}^{\tau}e^{\Abig(\tau-t)}\Bbig\yu(t)dt\right)\\
&=&[D^{*}\otimes I_{n}]\left(
[(DD^{*}+SS^{*})\otimes e^{A\tau}]\ex(0)+\int_{0}^{\tau}e^{\Abig(\tau-t)}\Bbig\yu(t)dt\right)\\
&=&[D^{*}\otimes I_{n}]\left(
[DD^{*}\otimes e^{A\tau}]\ex(0)+\int_{0}^{\tau}e^{\Abig(\tau-t)}\Bbig\yu(t)dt\right)+
[D^{*}SS^{*}\otimes e^{A\tau}]\ex(0)\\
&=&[D^{*}\otimes I_{n}]\left(
[I_{q}\otimes e^{A\tau}][DD^{*}\otimes I_{n}]\ex(0)+\int_{0}^{\tau}e^{\Abig(\tau-t)}\Bbig\yu(t)dt\right)\\
&=&[D^{*}\otimes I_{n}]\left(
e^{\Abig\tau}\xi+\int_{0}^{\tau}e^{\Abig(\tau-t)}\Bbig\yu(t)dt\right)\\
&=&0\,.
\end{eqnarray*}
That is, $\|\ex(\tau)\|_{\setS_{n}}=0$.
The array $[A,\,(B_{::})]$ then has to be controllable because $\ex(0)$ was arbitrary.
\end{proof}

\begin{theorem}\label{thm:con}
The following are equivalent.
\begin{enumerate}
\item The array $[A,\,(B_{::})]$ is controllable. \item The
controllability matrix $\Wbig$ is connected. \item All the
matrices
${\rm inc}\,(V_{1}^{*}B_{::}),\,{\rm inc}\,(V_{2}^{*}B_{::}),\,\ldots,\,{\rm inc}\,(V_{m}^{*}B_{::})$
are connected.
\end{enumerate}
\end{theorem}

\begin{proof}
{1 $\Longleftrightarrow$ 2.} By definition $\Wbig$ connected means ${\rm range}\,\Wbig\supset\setS_{n}^{\perp}$, which is equivalent to the controllability of the
array $[A,\,(B_{::})]$ by Lemma~\ref{lem:con}.

{\em 2$\implies$3.} Suppose for some
$\kappa\in\{1,\,2,\,\ldots,\,m\}$ the matrix ${\rm inc}\,(V_{\kappa}^{*}B_{::})$ is not connected. Recall ${\rm inc}\,(V_{\kappa}^{*}B_{::})=[I_{q}\otimes V_{\kappa}^{*}]\Bbig$. Now, ${\rm range}\,[I_{q}\otimes V_{\kappa}^{*}]\Bbig
\not\supset\setS_{d_{\kappa}}^{\perp}$ implies ${\rm null}\,\Bbig^{*}[I_{q}\otimes V_{\kappa}]\not\subset\setS_{d_{\kappa}}$. Therefore there exists a vector
$f\notin \setS_{d_{\kappa}}$ satisfying $\Bbig^{*}[I_{q}\otimes V_{\kappa}]f=0$. Define $\xi=[I_{q}\otimes V_{\kappa}]f$. We have $\xi\notin\setS_{n}$ because $f\notin \setS_{d_{\kappa}}$ and $V_{\kappa}$
is full column rank. Recall $A^{*}V_{\kappa}=\mu_{\kappa}V_{\kappa}$. This allows us
to write $\Abig^{*}[I_{q}\otimes V_{\kappa}]=\mu_{\kappa}[I_{q}\otimes V_{\kappa}]$. We can
now proceed as
\begin{eqnarray*}
\Wbig^{*}\xi=\left[\begin{array}{c}
\Bbig^{*}\\
\Bbig^{*}\Abig^{*}\\
\vdots\\
\Bbig^{*}\Abig^{(n-1)*}
\end{array}
\right][I_{q}\otimes V_{\kappa}]f=
\left[\begin{array}{c}
\Bbig^{*}[I_{q}\otimes V_{\kappa}]f\\
\mu_{\kappa}\Bbig^{*}[I_{q}\otimes V_{\kappa}]f\\
\vdots\\
\mu_{\kappa}^{n-1}\Bbig^{*}[I_{q}\otimes V_{\kappa}]f
\end{array}
\right]=0\,.
\end{eqnarray*}
Since $\xi\notin\setS_{n}$ we deduce ${\rm null}\,\Wbig^{*}\not\subset\setS_{n}$. Hence
${\rm range}\,\Wbig\not\supset\setS_{n}^{\perp}$, i.e., $\Wbig$ is not connected.

{\em 3$\implies$1.} Suppose the array $[A,\,(B_{::})]$ is not controllable. Then by Lemma~\ref{lem:con}
the pair $[\Abigr,\,\Bbigr]$ is not controllable. Thanks to PBH test
this implies that there exists an eigenvector $\eta\in(\Complex^{n})^{q-1}$ of
$\Abigr^{*}$ satisfying $\Bbigr^{*}\eta=0$. Note that $\Abigr^{*}$ and $A^{*}$ share the same eigenvalues. Therefore for some $\kappa\in\{1,\,2,\,\ldots,\,m\}$ we have $[{\mathbf A}_{\rm r}^{*}-\mu_{\kappa} I_{(q-1)n}]\eta=0$. Since ${\rm null}\,[\Abigr^{*}-\mu_{\kappa}I_{(q-1)n}]={\rm range}\,[I_{q-1}\otimes V_{\kappa}]$ there exists $h\in(\Complex^{d_{\kappa}})^{q-1}$
satisfying $[I_{q-1}\otimes V_{\kappa}]h=\eta$. Note that
$\eta$ is nonzero because it is an eigenvector. Therefore $h\neq 0$.
Now define $g=[D\otimes I_{d_{\kappa}}]h$. We have $g\neq 0$ because $h$ is nonzero and $[D\otimes I_{d_{\kappa}}]$
is full column rank. Then the inclusion $g\in{\rm range}\,[D\otimes I_{d_{\kappa}}]=\setS_{d_{\kappa}}^{\perp}$
allows us to write $g\notin\setS_{d_{\kappa}}$. Moreover,
\begin{eqnarray*}
\Bbig^{*}[I_{q}\otimes V_{\kappa}]g
&=&\Bbig^{*}[I_{q}\otimes V_{\kappa}][D\otimes I_{d_{\kappa}}]h\\
&=&\Bbig^{*}[D\otimes I_{n}][I_{q-1}\otimes V_{\kappa}]h\\
&=&\Bbigr^{*}\eta\\
&=&0\,.
\end{eqnarray*}
Hence ${\rm null}\,\Bbig^{*}[I_{q}\otimes V_{\kappa}]\not\subset\setS_{d_{\kappa}}$, yielding
${\rm range}\,[I_{q}\otimes V_{\kappa}^{*}]\Bbig\not\supset\setS_{d_{\kappa}}^{\perp}$, i.e.,
${\rm inc}\,(V_{\kappa}^{*}B_{::})$ is not connected.
\end{proof}
\vspace{0.12in}

{\em An example.} Using Theorem~\ref{thm:con} let us now study controllability of each of the two arrays of electrical oscillators shown in Fig.~\ref{fig:arrayscon}, where all inductances are 1H and all capacitances are 1F.

\begin{figure}[h]
\begin{center}
\includegraphics[scale=0.29]{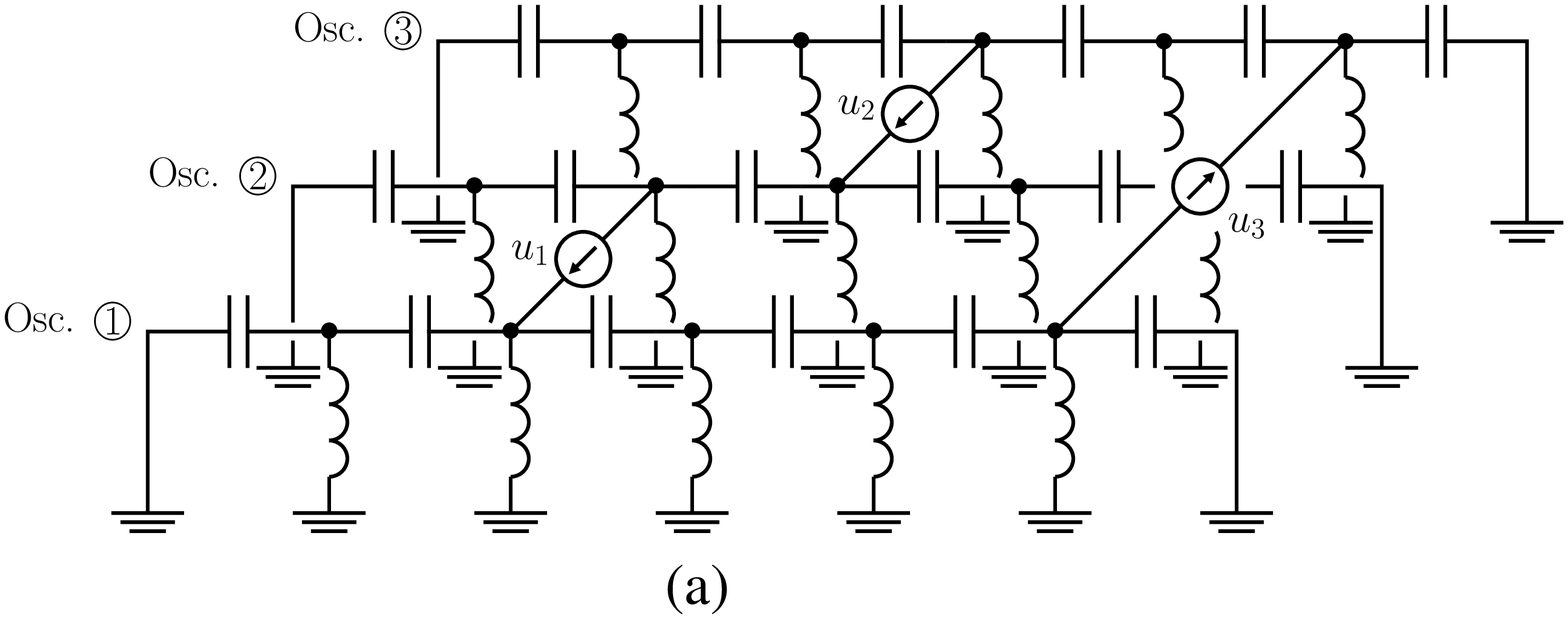}\includegraphics[scale=0.29]{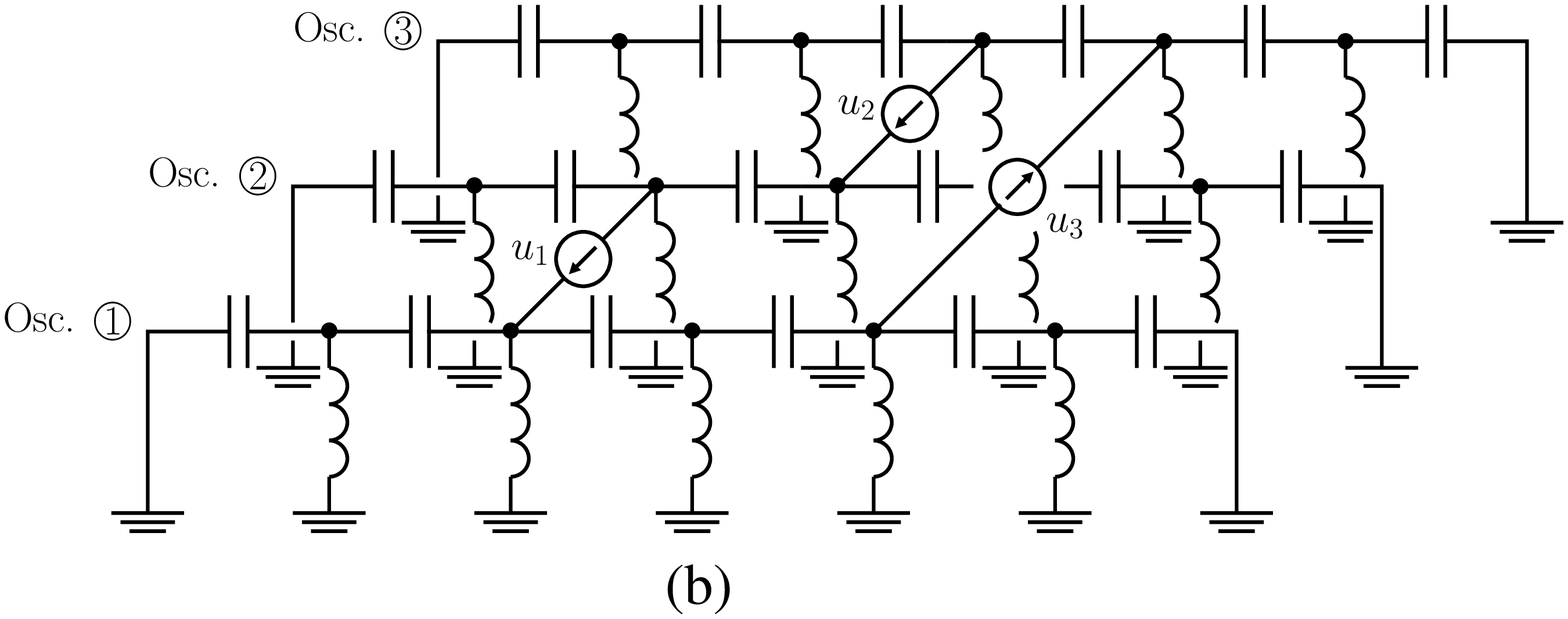}
\caption{Arrays of electrical oscillators.}
\label{fig:arrayscon}
\end{center}
\end{figure}

Let, for the $i$th oscillator, $y_{i}\in\Real^{5}$ be the vector of inductor currents and $v_{i}\in\Real^{5}$ be the vector of node voltages. We can then write
\begin{eqnarray*}
\begin{array}{rcl}
C{\dot v}_{1}+y_{1}&=&e_{2}u_{1}-e_{k}u_{3}\,,\\
C{\dot v}_{2}+y_{2}&=&e_{3}u_{2}-e_{2}u_{1}\,,\\
C{\dot v}_{3}+y_{3}&=&e_{k}u_{3}-e_{3}u_{2}\,,
\end{array}
\begin{array}{rcl}
L{\dot y}_{1}&=&v_{1}\\
L{\dot y}_{2}&=&v_{2}\\
L{\dot y}_{3}&=&v_{3}
\end{array}
\end{eqnarray*}
where $e_{2},\,e_{3},\,e_{k}\in\Real^{5}$, $e_{k}=e_{5}$ for the array in Fig.~\ref{fig:arrayscon}a, $e_{k}=e_{4}$ for the array in Fig.~\ref{fig:arrayscon}b, $L=I_{5}$, and
\begin{eqnarray*}
C=\left[\begin{array}{rrrrr}
2&-1&0&0&0\\
-1&2&-1&0&0\\
0&-1&2&-1&0\\
0&0&-1&2&-1\\
0&0&0&-1&2
\end{array}\right]\,.
\end{eqnarray*}
Now, defining the state of the $i$th oscillator as $x_{i}=[v_{i}^{*}\ y_{i}^{*}]^{*}$ we can rewrite the oscillator dynamics as
\begin{eqnarray*}
\left[\begin{array}{c}
{\dot x}_{1}\\
{\dot x}_{2}\\
{\dot x}_{3}
\end{array}\right]=
\left[\begin{array}{c}
Ax_{1}+B_{11}u_{1}-B_{33}u_{3}\\
Ax_{2}+B_{22}u_{2}-B_{11}u_{1}\\
Ax_{3}+B_{33}u_{3}-B_{22}u_{2}
\end{array}\right]=[I_{3}\otimes A]\left[\begin{array}{c}
{x}_{1}\\
{x}_{2}\\
{x}_{3}
\end{array}\right]+\underbrace{\left[\begin{array}{rrr}
B_{11}&0\ \ &-B_{33}\\
-B_{11}&B_{22}&0\ \ \\
0\ \ &-B_{22}&B_{33}
\end{array}\right]}_{\displaystyle{\rm inc}\,(B_{::})}\left[\begin{array}{c}
{u}_{1}\\
{u}_{2}\\
{u}_{3}
\end{array}\right]
\end{eqnarray*}
with
\begin{eqnarray*}
A=\left[\begin{array}{cc}
0& -C^{-1}\\
L^{-1}& 0
\end{array}\right]\,,\
B_{11}=\left[\begin{array}{c}
C^{-1}e_{2}\\
0
\end{array}\right]\,,\
B_{22}=\left[\begin{array}{c}
C^{-1}e_{3}\\
0
\end{array}\right]\,,\
B_{33}=\left[\begin{array}{c}
C^{-1}e_{k}\\
0
\end{array}\right]\,.
\end{eqnarray*}
The matrix $A^{*}\in\Real^{10\times 10}$ has five conjugate pairs of eigenvalues: $\mu_{1,2}=\pm j\sqrt{\tan(5\pi/12)}$, $\mu_{3,4}=\pm j1$, $\mu_{5,6}=\pm j\sqrt{1/2}$,
$\mu_{7,8}=\pm j\sqrt{1/3}$, $\mu_{9,10}=\pm j\sqrt{\tan(\pi/12)}$. Each eigenvalue $\mu_{\kappa}$
admits an eigenvector $V_{\kappa}\in\Complex^{10}$ and each eigenvector generates a class-$\G_{1}$
matrix ${\rm inc}\,(V_{\kappa}^{*}B_{::})\in\Complex^{3\times 3}$. Now, each ${\rm inc}\,(V_{\kappa}^{*}B_{::})$ is a (weighted\footnote{That is, each column is of the form $w_{ij}(e_{i}-e_{j})$ with $w_{ij}\in\Complex$.}) incidence matrix of a $3$-vertex graph, which is connected when the matrix ${\rm inc}\,(V_{\kappa}^{*}B_{::})$ is connected. These graphs for
each of the arrays in Fig.~\ref{fig:arrayscon} are given in Table~\ref{table:one}. Observe that all the graphs of the array in Fig.~\ref{fig:arrayscon}a are connected; whereas, for the array in Fig.~\ref{fig:arrayscon}b, the graph corresponding to the eigenvalue pair $\mu_{5,6}=\pm j\sqrt{1/2}$ is not connected. By Theorem~\ref{thm:con} therefore the array in Fig.~\ref{fig:arrayscon}a is controllable, but the array in Fig.~\ref{fig:arrayscon}b is not.

\renewcommand{\arraystretch}{1.5}
\begin{table}\caption{The graphs associated to the arrays in Fig.~\ref{fig:arrayscon}.}
\begin{center}
  \begin{tabular}{|c|c|c|c|c|c|}
    \hline
    Array  & ${\rm inc}\,(V_{1,2}^{*}B_{::})$ & ${\rm inc}\,(V_{3,4}^{*}B_{::})$ & ${\rm inc}\,(V_{5,6}^{*}B_{::})$ & ${\rm inc}\,(V_{7,8}^{*}B_{::})$ & ${\rm inc}\,(V_{9,10}^{*}B_{::})$\\ \hline
    \begin{tabular}{c}(a)\vspace{0.08in}\end{tabular} & \begin{tabular}{c}\includegraphics[scale=0.5]{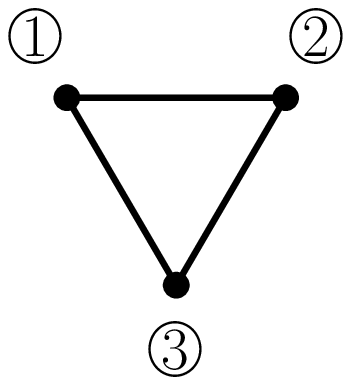}\vspace{-0.06in}\end{tabular} &
    \begin{tabular}{c}\includegraphics[scale=0.5]{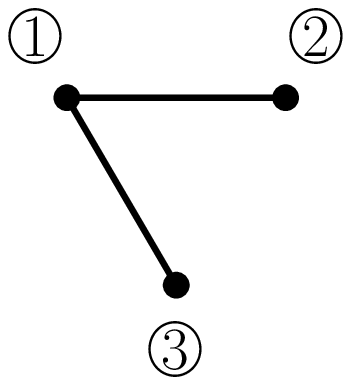}\vspace{-0.06in}\end{tabular} &
    \begin{tabular}{c}\includegraphics[scale=0.5]{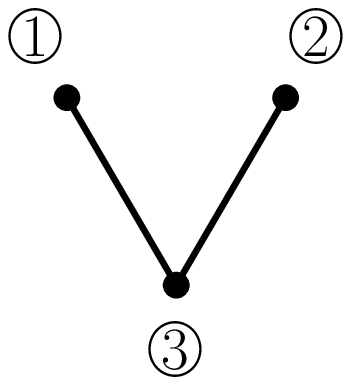}\vspace{-0.06in}\end{tabular} &
    \begin{tabular}{c}\includegraphics[scale=0.5]{graph2.eps}\vspace{-0.06in}\end{tabular} &
    \begin{tabular}{c}\includegraphics[scale=0.5]{graph1.eps}\vspace{-0.06in}\end{tabular} \\ \hline
    \begin{tabular}{c}(b)\vspace{0.08in}\end{tabular} & \begin{tabular}{c}\includegraphics[scale=0.5]{graph1.eps}\vspace{-0.06in}\end{tabular} &
    \begin{tabular}{c}\includegraphics[scale=0.5]{graph2.eps}\vspace{-0.06in}\end{tabular} &
    \begin{tabular}{c}\includegraphics[scale=0.5]{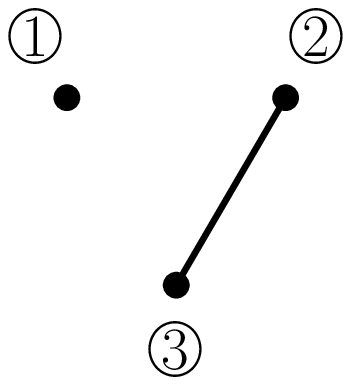}\vspace{-0.06in}\end{tabular} &
    \begin{tabular}{c}\includegraphics[scale=0.5]{graph2.eps}\vspace{-0.06in}\end{tabular} &
    \begin{tabular}{c}\includegraphics[scale=0.5]{graph1.eps}\vspace{-0.06in}\end{tabular} \\ \hline
  \end{tabular}
  \label{table:one}
\end{center}
\end{table}
\renewcommand{\arraystretch}{1}

\section{Positive controllability}\label{sec:poscontrol}

Recall that $\Abigr^{*}$ and $A^{*}$ share the same eigenvalues and ${\rm null}\,[\Abigr^{*}-\mu_{\kappa}I_{(q-1)n}]={\rm range}\,[I_{q-1}\otimes V_{\kappa}]$ for $\kappa=1,\,2,\,\ldots,\,m$. The below result is \cite[Cor.~1]{lindmark16} expressed in
our notation.

\begin{proposition}\label{prop:lindmark}
Consider the system ${\dot\ex}_{\rm r}=\Abigr\exr+\Bbigr\yur$. Suppose for each initial condition
$\exr(0)$ there exists an input signal $\yur:[0,\,\tau]\to\Real_{\geq 0}^{p}$ with some finite $\tau>0$
that achieves $\exr(\tau)=0$.
Then and only then the following two conditions simultaneously hold.
\begin{enumerate}
\item The pair $[\Abigr,\,\Bbigr]$ is controllable.
\item ${\rm cone}\,[I_{q-1}\otimes V_{\kappa}^{*}]\Bbigr=(\Real^{d_{\kappa}})^{q-1}$ for all $\mu_{\kappa}\in\Real$.
\end{enumerate}
\end{proposition}

A pair $[\Abigr,\,\Bbigr]$ is said to be {\em positively controllable} if it satisfies the conditions in Proposition~\ref{prop:lindmark}.

\begin{lemma}\label{lem:yin}
The following are equivalent.
\begin{enumerate}
\item The array $[A,\,(B_{::})]$ is positively controllable.
\item The pair $[\Abigr,\,\Bbigr]$ is positively controllable.
\end{enumerate}
\end{lemma}

\begin{proof}
{\em 1$\implies$2.} Suppose $[A,\,(B_{::})]$ is positively controllable. For the system~\eqref{eqn:bigsystem} this means each initial condition $\ex(0)$ can be driven to
$\setS_{n}$ in finite time with some nonnegative input signal. Consider the system ${\dot \ex}_{\rm r}=\Abigr\exr+\Bbigr\yur$.
Choose an arbitrary initial condition $\exr(0)\in(\Real^n)^{q-1}$. Set the initial condition
of the system~\eqref{eqn:bigsystem} as $\ex(0)=[D\otimes I_{n}]\exr(0)$. Then let $\yu:[0,\,\tau]\to\Real_{\geq 0}^{p}$ be
an input signal that yields $\|\ex(\tau)\|_{\setS_{n}}=0$ for some finite $\tau>0$.
In the proof of Lemma~\ref{lem:con} we discovered that the input signal $\yur(t)=\yu(t)$ for $t\in[0,\,\tau]$ renders $\exr(\tau)=0$. Hence the pair $[\Abigr,\,\Bbigr]$ must be positively controllable because $\exr(0)$ was arbitrary.

{\em 2$\implies$1.} Suppose $[\Abigr,\,\Bbigr]$ is positively controllable. Consider the system~\eqref{eqn:bigsystem}. Let $\ex(0)$ be an arbitrary initial condition. Then let
$\exr(0)=[D^{*}\otimes I_{n}]\ex(0)$ be the initial condition for the system ${\dot \ex}_{\rm r}=\Abigr\exr+\Bbigr\yur$. Thanks to positive controllability of $[\Abigr,\,\Bbigr]$ we can find an input signal $\yur:[0,\,\tau]\to\Real_{\geq 0}^{p}$ with finite $\tau>0$
that satisfies
\begin{eqnarray*}
e^{\Abigr\tau}\exr(0)+\int_{0}^{\tau}e^{\Abigr(\tau-t)}\Bbigr\yur(t)dt=0\,.
\end{eqnarray*}
Using this control signal to drive the system~\eqref{eqn:bigsystem}, i.e., $\yu(t)=\yur(t)$ for $t\in[0,\,\tau]$, we can write
\begin{eqnarray*}
[D^{*}\otimes I_{n}]\ex(\tau)
&=&[D^{*}\otimes I_{n}]\left(
e^{\Abig\tau}\ex(0)+\int_{0}^{\tau}e^{\Abig(\tau-t)}\Bbig\yu(t)dt\right)\\
&=&[D^{*}\otimes I_{n}]\left(
[I_{q}\otimes e^{A\tau}]\ex(0)+\int_{0}^{\tau}[I_{q}\otimes e^{A(\tau-t)}]\Bbig\yu(t)dt\right)\\
&=&[I_{q-1}\otimes e^{A\tau}][D^{*}\otimes I_{n}]\ex(0)+\int_{0}^{\tau}[I_{q-1}\otimes e^{A(\tau-t)}][D^{*}\otimes I_{n}]\Bbig\yu(t)dt\\
&=&e^{\Abigr\tau}\exr(0)+\int_{0}^{\tau}e^{\Abigr(\tau-t)}\Bbigr\yur(t)dt\\
&=&0\,.
\end{eqnarray*}
That is, $\|\ex(\tau)\|_{\setS_{n}}=0$.
Hence $[A,\,(B_{::})]$ has to be positively controllable because $\ex(0)$ was arbitrary.
\end{proof}

\begin{lemma}\label{lem:yang}
Let $\mu_{\kappa}\in\Real$. The following are equivalent.
\begin{enumerate}
\item ${\rm inc}\,(V_{\kappa}^{*}B_{::})$
is strongly connected.
\item ${\rm cone}\,[I_{q-1}\otimes V_{\kappa}^{*}]\Bbigr=(\Real^{d_{\kappa}})^{q-1}$.
\end{enumerate}
\end{lemma}

\begin{proof}
{\em 1$\implies$2.} Suppose ${\rm inc}\,(V_{\kappa}^{*}B_{::})$
is strongly connected. That is, ${\rm cone}\,[I_{q}\otimes V_{\kappa}^{*}]\Bbig\supset\setS_{d_{\kappa}}^{\perp}$. Let $g\in(\Real^{d_{\kappa}})^{q-1}$
be arbitrary. Define $f=[D\otimes I_{d_{\kappa}}]g$. Note that $f\in\setS_{d_{\kappa}}^{\perp}$.
Hence we can find a nonnegative vector $\alpha\in\Real_{\geq 0}^{p}$ satisfying $[I_{q}\otimes V_{\kappa}^{*}]\Bbig\alpha=f$. We can now write
\begin{eqnarray*}
[I_{q-1}\otimes V_{\kappa}^{*}]\Bbigr\alpha
&=&[I_{q-1}\otimes V_{\kappa}^{*}][D^{*}\otimes I_{n}]\Bbig\alpha\\
&=&[D^{*}\otimes I_{d_{\kappa}}][I_{q}\otimes V_{\kappa}^{*}]\Bbig\alpha\\
&=&[D^{*}\otimes I_{d_{\kappa}}]f\\
&=&[D^{*}\otimes I_{d_{\kappa}}][D\otimes I_{d_{\kappa}}]g\\
&=&[D^{*}D\otimes I_{d_{\kappa}}]g\\
&=&g\,.
\end{eqnarray*}
Hence $g\in{\rm cone}\,[I_{q-1}\otimes V_{\kappa}^{*}]\Bbigr$. Since $g$ was arbitrary we have ${\rm cone}\,[I_{q-1}\otimes V_{\kappa}^{*}]\Bbigr=(\Real^{d_{\kappa}})^{q-1}$.

{\em 2$\implies$1.} Suppose ${\rm cone}\,[I_{q-1}\otimes V_{\kappa}^{*}]\Bbigr=(\Real^{d_{\kappa}})^{q-1}$. Choose an arbitrary vector $h$ belonging to $\in\setS_{d_{\kappa}}^{\perp}$. Since $h\in\setS_{d_{\kappa}}^{\perp}$ we can find $b\in(\Real^{d_{\kappa}})^{q-1}$ satisfying $[D\otimes I_{d_{\kappa}}]b=h$. Then we can find
a nonnegative vector $\beta\in\Real_{\geq 0}^{p}$ satisfying $[I_{q-1}\otimes V_{\kappa}^{*}]\Bbigr\beta=b$. We can now write (recalling $[S^{*}\otimes I_{n}]\Bbig=0$)
\begin{eqnarray*}
[I_{q}\otimes V_{\kappa}^{*}]\Bbig\beta
&=&[I_{q}\otimes V_{\kappa}^{*}][(DD^{*}+SS^{*})\otimes I_{n}]\Bbig\beta\\
&=&[I_{q}\otimes V_{\kappa}^{*}][DD^{*}\otimes I_{n}]\Bbig\beta\\
&=&[D\otimes I_{d_{\kappa}}][I_{q-1}\otimes V_{\kappa}^{*}][D^{*}\otimes I_{n}]\Bbig\beta\\
&=&[D\otimes I_{d_{\kappa}}][I_{q-1}\otimes V_{\kappa}^{*}]\Bbigr\beta\\
&=&[D\otimes I_{d_{\kappa}}]b\\
&=&h\,.
\end{eqnarray*}
This shows that ${\rm cone}\,[I_{q}\otimes V_{\kappa}^{*}]\Bbig\supset\setS_{d_{\kappa}}^{\perp}$.
\end{proof}
\vspace{0.12in}

Proposition~\ref{prop:lindmark}, Lemma~\ref{lem:con}, Lemma~\ref{lem:yin}, and Lemma~\ref{lem:yang} yield:

\begin{theorem}\label{thm:poscon}
The array $[A,\,(B_{::})]$ is positively controllable if and only if the following two conditions simultaneously hold.
\begin{enumerate}
\item The array $[A,\,(B_{::})]$ is controllable.
\item ${\rm inc}\,(V_{\kappa}^{*}B_{::})$
is strongly connected for all $\mu_{\kappa}\in\Real$.
\end{enumerate}
\end{theorem}

\section{Pairwise controllability}\label{sec:paircon}

Let the integers $n_{1},\,n_{2},\,\ldots,\,n_{m}$ be the algebraic multiplicities of the distinct eigenvalues $\mu_{1},\,\mu_{2},\,\ldots,\,\mu_{m}$, respectively. Hence the characteristic polynomial of $A^{*}$ reads $(s-\mu_{1})^{n_{1}}(s-\mu_{2})^{n_{2}}\cdots(s-\mu_{m})^{n_{m}}$ with $n_{1}+n_{2}+\cdots+n_{m}=n$. For $\kappa=1,\,2,\,\ldots,\,m$ let $U_{\kappa}\in\Complex^{n\times n_{\kappa}}$ be a full column rank matrix satisfying ${\rm range}\,U_{\kappa}={\rm null}\,[A^{*}-\mu_{\kappa}I_{n}]^{n_{\kappa}}$. Without loss of generality we let $U_{\kappa}$ be real when the corresponding eigenvalue $\mu_{\kappa}$ is real. Since ${\rm range}\,U_{\kappa}$ is invariant with respect to $A^{*}$, for each $\kappa$ there exists a square matrix  $A_{\kappa}\in\Complex^{n_{\kappa}\times n_{\kappa}}$ satisfying $A^{*}U_{\kappa}=U_{\kappa}A_{\kappa}^{*}$. Note that each $A_{\kappa}^{*}$ has a single distinct eigenvalue. In other words, $(s-\mu_{\kappa})^{n_{\kappa}}$ is the characteristic polynomial of
$A_{\kappa}^{*}$. Define $B_{i\sigma}^{[\kappa]}=U_{\kappa}^{*}B_{i\sigma}$ and construct the following controllability matrix
\begin{eqnarray*}
W_{i\sigma}^{[\kappa]}=[B_{i\sigma}^{[\kappa]}\ A_{\kappa}B_{i\sigma}^{[\kappa]}\ \cdots\ A_{\kappa}^{n_{\kappa}-1}B_{i\sigma}^{[\kappa]}]\,.
\end{eqnarray*}

\begin{lemma}\label{lem:on}
We have
${\rm null}\,[{\rm inc}\,(W_{::}^{[\kappa]})]^{*}={\rm null}\,\Wbig^{*}[I_{q}\otimes U_{\kappa}]$ for all $\kappa=1,\,2,\,\ldots,\,m$.
\end{lemma}

\begin{proof}
Let $\Bbig_{\kappa}=[I_{q}\otimes U_{\kappa}^{*}]\Bbig$ and $\Abig_{\kappa}=[I_{q}\otimes A_{\kappa}]$. Let $\Wbig_{\kappa}=[\Bbig_{\kappa}\ \Abig_{\kappa}\Bbig_{\kappa}\ \cdots\ \Abig_{\kappa}^{n_{\kappa}-1}\Bbig_{\kappa}]$ and define its augmented version as ${\widehat\Wbig}_{\kappa}=[\Bbig_{\kappa}\ \Abig_{\kappa}\Bbig_{\kappa}\ \cdots\ \Abig_{\kappa}^{n-1}\Bbig_{\kappa}]$. Observe that $A^{*}U_{\kappa}=U_{\kappa}A_{\kappa}^{*}$ implies $A^{r*}U_{\kappa}=U_{\kappa}A_{\kappa}^{r*}$ for any nonnegative integer $r$. We can therefore write
\begin{eqnarray*}
\Abig^{r*}[I_{q}\otimes U_{\kappa}]
&=&[I_{q}\otimes A^{r*}][I_{q}\otimes U_{\kappa}]\\
&=&[I_{q}\otimes A^{r*}U_{\kappa}]\\
&=&[I_{q}\otimes U_{\kappa}A_{\kappa}^{r*}]\\
&=&[I_{q}\otimes U_{\kappa}][I_{q}\otimes A_{\kappa}^{r*}]\\
&=&[I_{q}\otimes U_{\kappa}]\Abig_{\kappa}^{r*}\,.
\end{eqnarray*}
Whence follows
\begin{eqnarray}\label{eqn:nade}
\Wbig^{*}[I_{q}\otimes U_{\kappa}]
&=&\left[\begin{array}{c}
\Bbig^{*}\\
\Bbig^{*}\Abig^{*}\\
\vdots\\
\Bbig^{*}\Abig^{(n-1)*}
\end{array}
\right][I_{q}\otimes U_{\kappa}]=\left[\begin{array}{c}
\Bbig^{*}[I_{q}\otimes U_{\kappa}]\\
\Bbig^{*}[I_{q}\otimes U_{\kappa}]\Abig_{\kappa}^{*}\\
\vdots\\
\Bbig^{*}[I_{q}\otimes U_{\kappa}]\Abig_{\kappa}^{(n-1)*}
\end{array}
\right]
=\left[\begin{array}{c}
\Bbig_{\kappa}^{*}\\
\Bbig_{\kappa}^{*}\Abig_{\kappa}^{*}\\
\vdots\\
\Bbig_{\kappa}^{*}\Abig_{\kappa}^{(n-1)*}
\end{array}
\right]\nonumber\\
&=&{\widehat\Wbig}_{\kappa}^{*}\,.
\end{eqnarray}
Since $\Abig_{\kappa}^{*}$ satisfies the characteristic equation of $A_{\kappa}^{*}$ (which is of order $n_{\kappa}\leq n$) we have
\begin{eqnarray}\label{eqn:pink}
{\rm null}\,{\widehat\Wbig}_{\kappa}^{*}={\rm null}\,\Wbig_{\kappa}^{*}\,.
\end{eqnarray}
Moreover, carrying out the multiplication explicitly one can obtain the identity $\Wbig_{\kappa}\Wbig_{\kappa}^{*}=[{\rm inc}\,(W_{::}^{[\kappa]})]\times[{\rm inc}\,(W_{::}^{[\kappa]})]^{*}$ which yields
\begin{eqnarray}\label{eqn:lemo}
{\rm null}\,\Wbig_{\kappa}^{*}={\rm null}\,[{\rm inc}\,(W_{::}^{[\kappa]})]^{*}\,.
\end{eqnarray}
Combining \eqref{eqn:nade}, \eqref{eqn:pink}, and \eqref{eqn:lemo} yields the result.
\end{proof}

\begin{lemma}\label{lem:onade}
The following are equivalent.
\begin{enumerate}
\item The array $[A,\,(B_{::})]$ is $(k,\,\ell)$-controllable.
\item ${\rm range}\,\Wbig\supset{\rm range}\,[(e_{k}-e_{\ell})\otimes I_{n}]$.
\end{enumerate}
\end{lemma}

\begin{proof}
{\em 1$\implies$2.} Suppose $[A,\,(B_{::})]$ is $(k,\,\ell)$-controllable. Consider the system~\eqref{eqn:bigsystem}. Choose an arbitrary initial condition $\ex(0)=[x_{1}(0)^{*}\ x_{2}(0)^{*}\ \cdots\ x_{q}(0)^{*}]^{*}\in{\rm range}\,[(e_{k}-e_{\ell})\otimes I_{n}]$. Clearly, we have $x_{i}(0)=0$ for $i\neq k,\,\ell$. Moreover, $x_{\rm av}(0)=0$. Let now $\yu:[0,\,\tau]\to\Real^{p}$ be some control signal (with $\tau>0$ finite) that achieves $x_{k}(\tau)-x_{\ell}(\tau)=0$ and $x_{i}(\tau)=e^{A\tau}x_{i}(0)=0$ for $i\neq k,\,\ell$. Such $\yu$ exists because $[A,\,(B_{::})]$ is $(k,\,\ell)$-controllable. Recall that ${\dot x}_{\rm av}=Ax_{\rm av}$ because the actuation is relative~\eqref{eqn:relative}. Hence $x_{\rm av}(\tau)=e^{A\tau}x_{\rm av}(0)=0$. We can therefore write
\begin{eqnarray*}
x_{k}(\tau)+x_{\ell}(\tau)
&=&x_{k}(\tau)+x_{\ell}(\tau)+\sum_{i\neq k,\ell}x_{i}(\tau)\\
&=&q x_{\rm av}(\tau)\\
&=&0\,.
\end{eqnarray*}
Recall $x_{k}(\tau)-x_{\ell}(\tau)=0$. Hence $x_{k}(\tau)+x_{\ell}(\tau)=0$ means $x_{k}(\tau)=x_{\ell}(\tau)=0$ yielding $\ex(\tau)=0$. This implies for the system~\eqref{eqn:bigsystem} that any initial condition
from the set ${\rm range}\,[(e_{k}-e_{\ell})\otimes I_{n}]$ can be driven to the origin in finite
time. This set then must be contained in the controllable subspace. In other words, ${\rm range}\,\Wbig\supset{\rm range}\,[(e_{k}-e_{\ell})\otimes I_{n}]$.

{\em 2$\implies$1.} Suppose ${\rm range}\,[(e_{k}-e_{\ell})\otimes I_{n}]$ is contained
in ${\rm range}\,\Wbig$, the controllable subspace of the system~\eqref{eqn:bigsystem}.
Let us be given an arbitrary initial condition $\ex(0)=[x_{1}(0)^{*}\ x_{2}(0)^{*}\ \cdots\ x_{q}(0)^{*}]^{*}$. Let $z=(x_{k}(0)-x_{\ell}(0))/2$. Then let $\hat\yu:[0,\,\tau]\to\Real^{p}$
be some control signal (with $\tau>0$ finite) that steers the initial condition
$\hat\ex(0)=(e_{k}-e_{\ell})\otimes z$ to the origin $\hat\ex(\tau)=0$. Such $\hat\yu$
exists because $\hat\ex(0)\in{\rm range}\,[(e_{k}-e_{\ell})\otimes I_{n}]$ belongs to the controllable subspace. In particular, we can write
\begin{eqnarray*}
\int_{0}^{\tau}e^{\Abig(\tau-t)}\Bbig\hat\yu(t)dt
=-e^{\Abig\tau}\hat\ex(0)\,.
\end{eqnarray*}
Using the same input for the initial condition $\ex(0)$ yields
\begin{eqnarray*}\label{eqn:tuna}
\ex(\tau)
&=&e^{\Abig\tau}\ex(0)+\int_{0}^{\tau}e^{\Abig(\tau-t)}\Bbig\hat\yu(t)dt\\
&=&e^{\Abig\tau}\ex(0)-e^{\Abig\tau}\hat\ex(0)\\
&=&e^{\Abig\tau}(\ex(0)-\hat\ex(0))\\
&=&[I_{q}\otimes e^{A\tau}](\ex(0)-\frac{1}{2}[(e_{k}-e_{\ell})\otimes (x_{k}(0)-x_{\ell}(0))])\,.
\end{eqnarray*}
It is now easy to verify $x_{k}(\tau)=x_{\ell}(\tau)=e^{A\tau}(x_{k}(0)+x_{\ell}(0))/2$ and $x_{i}(\tau)=e^{A\tau}x_{i}(0)$
for $i\neq k,\,\ell$.
\end{proof}

\begin{theorem}\label{thm:klcon}
The following are equivalent.
\begin{enumerate}
\item The array $[A,\,(B_{::})]$ is $(k,\,\ell)$-controllable.
\item The controllability matrix $\Wbig$ is $(k,\,\ell)$-connected.
\item All the matrices ${\rm inc}\,(W_{::}^{[1]}),\,{\rm inc}\,(W_{::}^{[2]}),\,\ldots,\,{\rm inc}\,(W_{::}^{[m]})$ are $(k,\,\ell)$-connected.
\end{enumerate}
\end{theorem}

\begin{proof}
{\em 1 $\Longleftrightarrow$ 2.} By Lemma~\ref{lem:onade}.

{\em 2$\implies$3.} Suppose for some $\kappa\in\{1,\,2,\,\ldots,\,m\}$ the matrix ${\rm inc}\,(W_{::}^{[\kappa]})$ is not $(k,\,\ell)$-connected. That is, ${\rm range}\,[{\rm inc}\,(W_{::}^{[\kappa]})]\not\supset{\rm range}\,[(e_{k}-e_{\ell})\otimes I_{n_{\kappa}}]$. Then ${\rm null}\,\Wbig^{*}[I_{q}\otimes U_{\kappa}]\not\subset
{\rm null}\,[(e_{k}-e_{\ell})^{*}\otimes I_{n_{\kappa}}]$ by Lemma~\ref{lem:on}. Therefore
there exists a vector $f\notin{\rm null}\,[(e_{k}-e_{\ell})^{*}\otimes I_{n_{\kappa}}]$ satisfying
$\Wbig^{*}[I_{q}\otimes U_{\kappa}]f=0$. Define $\xi=[I_{q}\otimes U_{\kappa}]f$. We have
$\xi\notin{\rm null}\,[(e_{k}-e_{\ell})^{*}\otimes I_{n}]$ because $f\notin{\rm null}\,[(e_{k}-e_{\ell})^{*}\otimes I_{n_{\kappa}}]$ and $U_{\kappa}$ is full column rank.
Hence $\Wbig^{*}\xi=\Wbig^{*}[I_{q}\otimes U_{\kappa}]f=0$ implies
${\rm null}\,\Wbig^{*}\not\subset{\rm null}\,[(e_{k}-e_{\ell})^{*}\otimes I_{n}]$. Equivalently,
${\rm range}\,\Wbig\not\supset{\rm range}\,[(e_{k}-e_{\ell})\otimes I_{n}]$, i.e., $\Wbig$
is not $(k,\,\ell)$-connected.

{\em 3$\implies$2.} Suppose ${\rm range}\,\Wbig\not\supset{\rm range}\,[(e_{k}-e_{\ell})\otimes I_{n}]$. Hence ${\rm null}\,\Wbig^{*}\not\subset{\rm null}\,[(e_{k}-e_{\ell})^{*}\otimes I_{n}]$
and we can find
$\zeta\notin{\rm null}\,[(e_{k}-e_{\ell})^{*}\otimes I_{n}]$ satisfying $\Wbig^{*}\zeta=0$. Observe that ${\rm null}\,\Wbig^{*}$ is invariant with respect to $\Abig^{*}$. This allows us to write
$\Wbig^{*}e^{\Abig^{*}t}\zeta\equiv 0$. Note that the $n\times n$ matrix $[U_{1}\ U_{2}\ \cdots\ U_{m}]$ is nonsingular. This means that the $qn\times qn$ matrix $[[I_{q}\otimes U_{1}]\ [I_{q}\otimes U_{2}]\ \cdots\ [I_{q}\otimes U_{m}]]$ is nonsingular. Therefore we can find vectors
$f_{\kappa}\in(\Complex^{n_{\kappa}})^{q}$ satisfying
\begin{eqnarray*}
\zeta=\sum_{\kappa=1}^{m}\,[I_{q}\otimes U_{\kappa}]f_{\kappa}\,.
\end{eqnarray*}
Since we can write
\begin{eqnarray*}
\sum_{\kappa=1}^{m}\,[I_{q}\otimes U_{\kappa}][(e_{k}-e_{\ell})^{*}\otimes I_{n_{\kappa}}]f_{\kappa}
&=&[(e_{k}-e_{\ell})^{*}\otimes I_{n}]\sum_{\kappa=1}^{m}\,[I_{q}\otimes U_{\kappa}]f_{\kappa}\\
&=&[(e_{k}-e_{\ell})^{*}\otimes I_{n}]\zeta\\
&\neq&0
\end{eqnarray*}
we have to have
\begin{eqnarray}\label{eqn:kling}
[(e_{k}-e_{\ell})^{*}\otimes I_{n_{\kappa}}]f_{\kappa}\neq 0
\end{eqnarray}
for some $\kappa$. Recall $A^{*}U_{\kappa}=U_{\kappa}A_{\kappa}^{*}$, which implies
$e^{A^{*}t}U_{\kappa}=U_{\kappa}e^{A_{\kappa}^{*}t}$.
We can therefore write
\begin{eqnarray}\label{eqn:spar}
\Wbig^{*}\sum_{\kappa=1}^{m}\,[I_{q}\otimes U_{\kappa}e^{A_{\kappa}^{*}t}]f_{\kappa}
&=& \Wbig^{*}\sum_{\kappa=1}^{m}\,[I_{q}\otimes e^{A^{*}t}U_{\kappa}]f_{\kappa}\nonumber\\
&=& \Wbig^{*}[I_{q}\otimes e^{A^{*}t}]\sum_{\kappa=1}^{m}\,[I_{q}\otimes U_{\kappa}]f_{\kappa}\nonumber\\
&=& \Wbig^{*}e^{\Abig^{*}t}\zeta\nonumber\\
&\equiv&0\,.
\end{eqnarray}
Now, since no two matrices $A_{\kappa}^{*},\,A_{\nu}^{*}$ ($\kappa\neq\nu$) share a common eigenvalue, the set of mappings $\{t\mapsto [I_{q}\otimes U_{\kappa}e^{A_{\kappa}^{*}t}]f_{\kappa}:f_{\kappa}\neq 0\}$ is linearly independent. Hence
\eqref{eqn:spar} implies $\Wbig^{*}[I_{q}\otimes U_{\kappa}e^{A_{\kappa}^{*}t}]f_{\kappa}\equiv 0$, which in turn implies
\begin{eqnarray}\label{eqn:usta}
\Wbig^{*}[I_{q}\otimes U_{\kappa}]f_{\kappa}=0
\end{eqnarray}
for all $\kappa$. Combining \eqref{eqn:kling} and \eqref{eqn:usta} allows us to claim
${\rm null}\,\Wbig^{*}[I_{q}\otimes U_{\kappa}]\not\subset{\rm null}\,[(e_{k}-e_{\ell})^{*}\otimes I_{n_{\kappa}}]$ for some $\kappa$. Then by Lemma~\ref{lem:on} we have ${\rm null}\,[{\rm inc}\,(W_{::}^{[\kappa]})]^{*}\not\subset{\rm null}\,[(e_{k}-e_{\ell})^{*}\otimes I_{n_{\kappa}}]$. Therefore ${\rm range}\,[{\rm inc}\,(W_{::}^{[\kappa]})]\not\supset{\rm range}\,[(e_{k}-e_{\ell})\otimes I_{n_{\kappa}}]$. That is, ${\rm inc}\,(W_{::}^{[\kappa]})$
is not $(k,\,\ell)$-connected.
\end{proof}
\vspace{0.12in}

Note that the eigenvector test for controllability stated in Theorem~\ref{thm:con} cannot be extended to $(k,\,\ell)$-controllability in general. In particular, an array may fail to be $(k,\,\ell)$-controllable even if all the
matrices ${\rm inc}\,(V_{1}^{*}B_{::}),\,{\rm inc}\,(V_{2}^{*}B_{::}),\,\ldots,\,{\rm inc}\,(V_{m}^{*}B_{::})$ are
$(k,\,\ell)$-connected. A
counterexample is as follows. Consider an array of $q=3$ fourth order systems with
\begin{eqnarray*}
A = \left[\begin{array}{cccc}
0&0&0&0\\
1&0&0&0\\
0&0&0&0\\
0&0&1&0
\end{array}\right]
\quad\mbox{and}\quad
{\rm inc}\,(B_{::}) = \left[\begin{array}{rrr}
     0  &   0  &   0\\
     0  &   0  &  -1\\
     1  &   0  &  -1\\
     0  &   0  &   0\\
     0  &   1  &   0\\
     0  &   0  &   0\\
    -1  &   0  &   0\\
     0  &   0  &   0\\
     0  &  -1  &   0\\
     0  &   0  &   1\\
     0  &   0  &   1\\
     0  &   0  &   0
\end{array}\right]\,.
\end{eqnarray*}
The matrix $A^{*}$ has no eigenvalue other than $\mu_{1}=0$ for which
\begin{eqnarray*}
V_{1} = \left[\begin{array}{cc}
1&0\\
0&0\\
0&1\\
0&0
\end{array}\right]\,.
\end{eqnarray*}
The matrix $\Wbig$ is not $(2,\,3)$-connected for this array. Therefore by Theorem~\ref{thm:klcon}, the array is not $(2,\,3)$-controllable. Despite this lack of $(2,\,3)$-controllability,
the matrix ${\rm inc}\,(V_{1}^{*}B_{::})$ however is $(2,\,3)$-connected.

\section{Positive pairwise controllability}\label{sec:pairposcon}

For the system~\eqref{eqn:bigsystem} let $\R\subset(\Real^{n})^{q}$ be the set of points positively reachable (in finite time) from the
origin, i.e.,
\begin{eqnarray*}
\R=\left\{\xi:\xi=\int_{0}^{\tau}e^{\Abig(\tau-t)}\Bbig\yu(t)dt\ \mbox{for some}\ \yu:[0,\,\tau]\to\Real_{\geq 0}^{p}\ \mbox{with}\ \tau\geq 0\right\}\,.
\end{eqnarray*}
The set $\R$ is a convex cone. The {\em polar} of $\R$ is denoted by $\R^{\circ}$, which itself is a convex cone in $(\Real^{n})^{q}$, and defined as $\R^{\circ}=\{\eta:\xi^{*}\eta\leq 0\ \mbox{for all}\ \xi\in\R\}$. Note that we can write
\begin{eqnarray*}
\R^{\circ}=\{\eta:\Bbig^{*}e^{\Abig^{*}t}\eta\leq 0\ \mbox{for all}\ t\geq 0\}\,.
\end{eqnarray*}
The cone $\R^{\circ}$ is closed and the polar $\R^{\circ\circ}$ of $\R^{\circ}$ equals ${\rm cl}\,\R$, the closure of $\R$ \cite{rockafellar70}.

Let $\Lambda_{\kappa}=[A_{\kappa}^{*}-\mu_{\kappa}I_{n_{\kappa}}]^{*}$ for $\kappa=1,\,2,\,\ldots,\,m$. Since $\mu_{\kappa}$ is the only eigenvalue of $A_{\kappa}^{*}$, the matrix $\Lambda_{\kappa}$ has a single distinct eigenvalue at the origin, i.e., it is nilpotent. In particular, $\Lambda_{\kappa}^{n_{\kappa}}=0$. Recall $B_{i\sigma}^{[\kappa]}=U_{\kappa}^{*}B_{i\sigma}$ and ${\rm range}\,U_{\kappa}={\rm null}\,[A^{*}-\mu_{\kappa}I_{n}]^{n_{\kappa}}$. Let us now define
\begin{eqnarray*}
Q_{i\sigma}^{[\kappa]}=[B_{i\sigma}^{[\kappa]}\ \Lambda_{\kappa}B_{i\sigma}^{[\kappa]}\ \cdots\ \Lambda_{\kappa}^{n_{\kappa}-1}B_{i\sigma}^{[\kappa]}]\,.
\end{eqnarray*}
With slight abuse of notation we also let $B_{:\sigma}^{[\kappa]}=[B_{1\sigma}^{[\kappa]*}\ B_{2\sigma}^{[\kappa]*}\ \cdots\ B_{q\sigma}^{[\kappa]*}]^{*}$ and $B_{:\sigma}=[B_{1\sigma}^{*}\ B_{2\sigma}^{*}\ \cdots\ B_{q\sigma}^{*}]^{*}$. Without loss of generality we henceforth assume
\begin{itemize}
\item ${\rm Re}\,\mu_{1}\geq{\rm Re}\,\mu_{2}\geq\cdots\geq{\rm Re}\,\mu_{m}$ and
\item ${\rm Re}\,\mu_{\kappa}=\mu_{\nu}$ ($\kappa\neq\nu$) implies $\kappa>\nu$,
\end{itemize}
where ${\rm Re}\,\mu_{\kappa}$ denotes the real part of $\mu_{\kappa}$. For each $\mu_{\kappa}\in\Real$, let $\Q_{\kappa}\subset(\Real^{n_{\kappa}})^q$ be the largest subspace contained in ${\rm cone}\,[{\rm inc}\,(Q_{:\sigma}^{[\kappa]})_{\sigma\in\I_{\kappa}}]$ where the index sets $\I_{\kappa}\subset\{1,\,2,\,\ldots,\,p\}$ are constructed as follows. $\I_{1}=\{1,\,2,\,\ldots,\,p\}$ and
\begin{eqnarray*}
\I_{\kappa+1}=\I_\kappa\setminus \I_\kappa^{-}
\end{eqnarray*}
with
\begin{eqnarray*}
\I_{\kappa}^{-}=\left\{\begin{array}{cl} \{\sigma\in\I_{\kappa}:[I_{q}\otimes\Lambda_{\kappa}^{r}]B_{:\sigma}^{[\kappa]}
\notin\Q_{\kappa}\ \mbox{for some}\ r\in\{0,\,1,\,\ldots,\,n_{\kappa}-1\}\}\ &\mbox{if}\ \mu_{\kappa}\in\Real\\
\emptyset\ &\mbox{if}\ \mu_{\kappa}\notin\Real\,.
\end{array}
\right.
\end{eqnarray*}

\begin{lemma}\label{lem:pivotal}
Let $N\in\Real^{n\times n}$ be a nilpotent matrix ($N^{n}=0$) and $C\in\Real^{p\times n}$. Define
the convex cones $\N=\{\eta\in\Real^{n}:Ce^{Nt}\eta\leq 0\ \mbox{for all}\ t\geq 0\}$ and
$\M=\{\eta\in\Real^{n}:CN^{r}\eta\leq 0\ \mbox{for all}\ r\in\{0,\,1,\,\ldots,\,n-1\}\}$. Let
$\D\subset\Real^{n}$ be the smallest subspace containing $\M$. Then $\N\subset\D$.
\end{lemma}

\begin{proof}
Let us first find an explicit expression of the subspace $\D$ in terms of our parameters. To this end let $C_{\sigma}\in\Real^{1\times n}$ denote the rows of $C$, i.e., $[C_{1}^{*}\ C_{2}^{*}\ \cdots\ C_{p}^{*}]^{*}=C$. Hence we can write $\M=\{\eta:C_{\sigma}N^{r}\eta\leq 0\ \mbox{for all}\ (\sigma,\,r)\}$. Define the set of pairs  $\setP^{-}\subset\{1,\,2,\,\ldots,\,p\}\times\{0,\,1,\,\ldots,\,n-1\}=:\setP$ as
$\setP^{-}=\{(\sigma,\,r):\M\cap\{\eta:C_{\sigma}N^{r}\eta<0\}\neq\emptyset\}$. Then let
$\setP_{0}=\setP\setminus\setP^{-}$ and $\M_{0}=\{\eta:C_{\sigma}N^{r}\eta\leq 0\ \mbox{for all}\ (\sigma,\,r)\in\setP_{0}\}$. Define the subspace $\D_{0}=\{\eta:C_{\sigma}N^{r}\eta=0\ \mbox{for all}\ (\sigma,\,r)\in\setP_{0}\}$. Clearly, $\D_{0}\subset\M_{0}$. We now claim
\begin{eqnarray}\label{eqn:DMO}
\D_{0}\supset\M_{0}\,.
\end{eqnarray}
The relation~\eqref{eqn:DMO} trivially holds for the extreme possibilities
$\setP_{0}=\emptyset$ or $\setP^{-}=\emptyset$.
For the case where neither $\setP_{0}$ nor $\setP^{-}$ is empty,
let us  establish our claim by contradiction. Suppose $\D_{0}\not\supset\M_{0}$.
Then we can find $\eta\in\M_{0}$ satisfying $C_{\bar\sigma}N^{\bar r}\eta<0$ for some $(\bar\sigma,\,\bar r)\in\setP_{0}$. Let now ${\mathcal F}=\{\eta_{1},\,\eta_{2},\,\ldots,\,\eta_{\gamma}\}\subset\M$
be a finite collection of vectors with the property that for each pair $(\sigma,\,r)\in\setP^{-}$ there exists some $\eta_{i}\in{\mathcal F}$ satisfying $C_{\sigma}N^{r}\eta_{i}<0$. Such ${\mathcal F}$ exists by how the set $\setP^{-}$ is defined. Let the scalars $\alpha_{1},\,\alpha_{2},\,\ldots,\,\alpha_{\gamma}$ satisfy $\alpha_{i}>0$ for all $i$ and
$\alpha_{1}+\alpha_{2}+\cdots+\alpha_{\gamma}=1$.
Then construct the convex combination of the vectors in ${\mathcal F}$ as $\hat\eta=\alpha_{1}\eta_{1}+\alpha_{2}\eta_{2}+\cdots+\alpha_{\gamma}\eta_{\gamma}$.
Since $\eta_{i}\in\M$ and $\M$ is convex, the new vector $\hat\eta$ belongs to $\M$. Moreover, since $\alpha_{i}$ are strictly positive, we have $C_{\sigma}N^{r}\hat\eta<0$ for all $(\sigma,\,r)\in\setP^{-}$. Now define for $\lambda\in(0,\,1)$
\begin{eqnarray}\label{eqn:etatilde}
\bar\eta=\lambda\eta+(1-\lambda)\hat\eta\,.
\end{eqnarray}
Note that $C_{\bar\sigma}N^{\bar r}\bar\eta<0$ for all $\lambda\in(0,\,1)$. Also, it is easy to check that by choosing $\lambda$ sufficiently small we can make $\bar\eta$ satisfy $C_{\sigma}N^{r}\bar\eta\leq 0$ for all $(\sigma,\,r)$, i.e., $\bar\eta\in\M$. Then
$C_{\bar\sigma}N^{\bar r}\bar\eta<0$ implies $(\bar\sigma,\,\bar r)\in\setP^{-}$, which contradicts $(\bar\sigma,\,\bar r)\in\setP_{0}$. Hence \eqref{eqn:DMO} holds true. In particular, since we also
have $\D_{0}\subset\M_{0}$, we can write $\D_{0}=\M_{0}$.

Next we show $\D=\M_{0}$. Choose an arbitrary vector $\eta$ that belongs to $\M_{0}$, i.e., $C_{\sigma}N^{r}\eta\leq0$ for all $(\sigma,\,r)\in\setP_{0}$. Let $\hat\eta\in\M$ be as before. That is, $C_{\sigma}N^{r}\hat\eta<0$ for all $(\sigma,\,r)\in\setP^{-}$ and $C_{\sigma}N^{r}\hat\eta\leq0$ for all $(\sigma,\,r)\in\setP_{0}$. Consider \eqref{eqn:etatilde}.
Choose $\lambda\in(0,\,1)$ small enough so that $C_{\sigma}N^{r}\bar\eta\leq0$ for all $(\sigma,\,r)$. Hence $\bar\eta\in\M$. Let us now rewrite \eqref{eqn:etatilde} as
\begin{eqnarray*}
\eta=\lambda^{-1}\bar\eta+(\lambda-1)\lambda^{-1}\hat\eta\,.
\end{eqnarray*}
Since both $\hat\eta$ and $\bar\eta$ belong to $\M$, we have $\hat\eta,\,\bar\eta\in\D$. Therefore $\eta\in\D$. Since $\eta$ was arbitrary we have to have $\D=\M_{0}$, i.e., $\M_{0}$ (or $\D_{0}$) is the smallest subspace containing $\M$.

So far we have not used the nilpotency of $N$. An obvious implication of $N^{n}=0$ is $N\M\subset\M$, i.e., the set $\M$ is invariant with respect to $N$.
This invariance imposes a special structure on $\setP^{-}$. To see that let some
pair $(\sigma,\,r)$ with $r\geq 1$ belong to $\setP^{-}$. That is, we can find some
$\eta\in\M$ satisfying $C_{\sigma}N^{r}\eta<0$. Then $N\eta\in\M$ yields
$(\sigma,\,r-1)\in\setP^{-}$ because $C_{\sigma}N^{r-1}(N\eta)<0$. Consequently, the complement
set $\setP_{0}$ enjoys (for $r\neq n-1$)
\begin{eqnarray}\label{eqn:P0}
(\sigma,\,r)\in\setP_{0}\implies (\sigma,\,r+1)\in\setP_{0}\,.
\end{eqnarray}
We are now ready to establish $\N\subset\D$. Suppose otherwise. Then we can find $\eta\in\N$
satisfying $\eta\notin\M_{0}$. Since $N^{n}=0$ the matrix exponential $e^{Nt}$ has finitely many terms. This allows us to write (for all $\sigma$)
\begin{eqnarray}\label{eqn:eNt}
C_{\sigma}\left(I_{n}+Nt+N^{2}\frac{t^{2}}{2}+\cdots+N^{n-1}\frac{t^{n-1}}{(n-1)!}\right)\eta\leq 0\ \mbox{for all}\ t\geq 0\,.
\end{eqnarray}
For each $\sigma$ let $r_{\sigma}\in\{0,\,1,\,\ldots,\,n\}$ be the smallest index
satisfying $C_{\sigma}N^{r}\eta=0$ for all $r\geq r_{\sigma}$. Considering \eqref{eqn:eNt}
as $t\to\infty$ we deduce
\begin{eqnarray}\label{eqn:disc}
C_{\sigma}N^{r_{\sigma}-1}\eta<0\ \ \mbox{if}\ \ r_{\sigma}\neq 0\,.
\end{eqnarray}
Hence we can write
\begin{eqnarray}\label{eqn:sion}
C_{\sigma}N^{r}\eta\leq0\ \ \mbox{for all}\ \ r\geq \max\,\{0,\,r_{\sigma}-1\}\,.
\end{eqnarray}
For each $\sigma$ this time let $\hat r_{\sigma}\in\{0,\,1,\,\ldots,\,n\}$ be the smallest
index satisfying $(\sigma,\,\hat r_{\sigma})\in\setP_{0}$ if exists. Otherwise (i.e., in case $(\sigma,\,n-1)\notin\setP_{0}$) let $\hat r_{\sigma}=n$. Now, due to the property~\eqref{eqn:P0} and $\eta\notin\M_{0}$ we have to have $r_{\sigma}>\hat r_{\sigma}$ for some $\sigma$. Therefore the integer $\rho=\max_{\sigma}\,(r_{\sigma}-\hat r_{\sigma})$ is positive. Define the vector $\hat\eta=N^{\rho-1}\eta$. Given some pair $(\sigma,\,\hat r)\in\setP_{0}$ define $r=\hat r+\rho-1$. Since $\hat r\geq\hat r_{\sigma}$ we have $r\geq r_{\sigma}-1$
and therefore $C_{\sigma}N^{r}\eta\leq0$ by \eqref{eqn:sion}. Hence
\begin{eqnarray*}
C_{\sigma}N^{\hat r}\hat\eta
&=&C_{\sigma}N^{\hat r+\rho-1}\eta\\
&=&C_{\sigma}N^{r}\eta\\
&\leq& 0
\end{eqnarray*}
meaning $\hat\eta\in\M_{0}$, for the pair $(\sigma,\,\hat r)\in\setP_{0}$ was arbitrary. Let now the index $\upsilon\in\{1,\,2,\,\ldots,\,p\}$ satisfy $r_{\upsilon}-\hat r_{\upsilon}=\rho$. Since $\rho$ is positive so is $r_{\upsilon}$.
As a result $C_{\upsilon}N^{r_{\upsilon}-1}\eta<0$ by \eqref{eqn:disc}. Then
\begin{eqnarray*}
C_{\upsilon}N^{\hat r_{\upsilon}}\hat\eta
&=&C_{\upsilon}N^{\hat r_{\upsilon}}N^{\rho-1}\eta\\
&=&C_{\upsilon}N^{r_{\upsilon}-1}\eta\\
&<&0
\end{eqnarray*}
meaning $\hat\eta\notin\D_{0}$ because $(\upsilon,\,\hat r_{\upsilon})\in\setP_{0}$ by definition. But this contradicts $\hat\eta\in\M_{0}$ because $\D_{0}=\M_{0}$.
\end{proof}

\begin{lemma}\label{lem:ur}
Let $\eta\in(\Real^{n})^{q}$ satisfy $\eta=\sum_{\kappa=1}^{m}[I_{q}\otimes U_{\kappa}]f_{\kappa}$
with $f_{\kappa}\in(\Complex^{n_{\kappa}})^{q}$ and $\Bbig^{*}e^{\Abig^{*}t}\eta\leq 0$ for all $t\geq 0$. There exists $\tau\geq0$ such that for each $\kappa$ we have ${\rm Re}\,(B_{:\sigma}^{[\kappa]*}[I_{q}\otimes e^{A_{\kappa}^{*}t}]f_{\kappa})\leq 0$ for all
$\sigma\in\I_{\kappa}$ and $t\geq \tau$.
\end{lemma}

\begin{proof}
Note that $\Bbig^{*}e^{\Abig^{*}t}\eta\leq 0$ for all $t\geq 0$ means for all $\sigma$ we have $B_{:\sigma}^{*}e^{\Abig^{*}t}\eta\leq 0$ for all $t\geq 0$. We can write
\begin{eqnarray}\label{eqn:bearing}
\sum_{\kappa=1}^{m}{\rm Re}\,(B_{:\sigma}^{[\kappa]*}[I_{q}\otimes e^{A_{\kappa}^{*}t}]f_{\kappa})
&=& {\rm Re}\,\left(\sum_{\kappa=1}^{m}B_{:\sigma}^{*}[I_{q}\otimes U_{\kappa}][I_{q}\otimes e^{A_{\kappa}^{*}t}]f_{\kappa}\right)\nonumber\\
&=& {\rm Re}\,\left(\sum_{\kappa=1}^{m}B_{:\sigma}^{*}[I_{q}\otimes U_{\kappa}e^{A_{\kappa}^{*}t}]f_{\kappa}\right)\nonumber\\
&=& {\rm Re}\,\left(\sum_{\kappa=1}^{m}B_{:\sigma}^{*}[I_{q}\otimes e^{A^{*}t}U_{\kappa}]f_{\kappa}\right)\nonumber\\
&=& {\rm Re}\,\left(\sum_{\kappa=1}^{m}B_{:\sigma}^{*}[I_{q}\otimes e^{A^{*}t}][I_{q}\otimes U_{\kappa}]f_{\kappa}\right)\nonumber\\
&=& {\rm Re}\,\left(\sum_{\kappa=1}^{m}B_{:\sigma}^{*}e^{\Abig^{*}t}[I_{q}\otimes U_{\kappa}]f_{\kappa}\right)\nonumber\\
&=& B_{:\sigma}^{*}e^{\Abig^{*}t}{\rm Re}\,\left(\sum_{\kappa=1}^{m}[I_{q}\otimes U_{\kappa}]f_{\kappa}\right)\nonumber\\
&=& B_{:\sigma}^{*}e^{\Abig^{*}t}\eta\nonumber\\
&\leq&0
\end{eqnarray}
where $U_{\kappa}e^{A_{\kappa}^{*}t}=e^{A^{*}t}U_{\kappa}$ follows from $U_{\kappa}A_{\kappa}^{*}=A^{*}U_{\kappa}$. Recall our ordering ${\rm Re}\,\mu_{1}\geq{\rm Re}\,\mu_{2}\geq\cdots\geq{\rm Re}\,\mu_{m}$ with the extra condition: whenever two distinct
indices $\kappa,\,\nu$ satisfy ${\rm Re}\,\mu_{\kappa}=\mu_{\nu}$ we have $\kappa>\nu$.
Since each $A_{\kappa}^{*}$ has a single distinct eigenvalue $\mu_{\kappa}$ inequality~\eqref{eqn:bearing} implies the existence of $\tau\geq 0$
such that for all $t\geq \tau$ we can write
\begin{eqnarray}\label{eqn:incidentallyzero}
\sum_{\kappa=1}^\nu{\rm Re}\,(B_{:\sigma}^{[\kappa]*}[I_{q}\otimes e^{A_{\kappa}^{*}t}]f_{\kappa})\leq 0
\end{eqnarray}
for all $\nu\in\{1,\,2,\,\ldots,\,m\}$ and $\sigma$. We now claim for all $\nu$ and $t\geq\tau$
\begin{eqnarray}\label{eqn:identicallyzero}
B_{:\sigma}^{[\kappa]*}[I_{q}\otimes e^{A_{\kappa}^{*}t}]f_{\kappa}=0\ \ \mbox{for all}\ \ \kappa\leq\nu\ \ \mbox{and}\ \ \sigma\in\I_{\nu}\setminus \I_{\nu}^{-}\,.
\end{eqnarray}
Let us establish our claim by induction. Suppose \eqref{eqn:identicallyzero} holds for some $\nu\in\{1,\,2,\,\ldots,\,m-1\}$. Then \eqref{eqn:incidentallyzero} allows us to write
for $\sigma\in\I_{\nu}\setminus \I_{\nu}^{-}$
\begin{eqnarray}\label{eqn:whynot}
{\rm Re}\,(B_{:\sigma}^{[\nu+1]*}[I_{q}\otimes e^{A_{\nu+1}^{*}t}]f_{\nu+1})
&=& \sum_{\kappa=1}^{\nu+1}{\rm Re}\,(B_{:\sigma}^{[\kappa]*}[I_{q}\otimes e^{A_{\kappa}^{*}t}]f_{\kappa})\nonumber\\
&\leq&0\,.
\end{eqnarray}
If $\mu_{\nu+1}\notin\Real$ then \eqref{eqn:whynot} implies $B_{:\sigma}^{[\nu+1]*}[I_{q}\otimes e^{A_{\nu+1}^{*}t}]f_{\nu+1}\equiv 0$, for otherwise, the oscillations would not let
the signal $t\mapsto {\rm Re}\,(B_{:\sigma}^{[\nu+1]*}[I_{q}\otimes e^{A_{\nu+1}^{*}t}]f_{\nu+1})$ remain nonpositive indefinitely. Hence
\eqref{eqn:identicallyzero} holds for $\nu+1$ because
\begin{eqnarray}\label{eqn:because}
\I_{\nu+1}\setminus \I_{\nu+1}^{-}\subset\I_{\nu+1}=\I_{\nu}\setminus \I_{\nu}^{-}\,.
\end{eqnarray}
Let us now consider the case $\mu_{\nu+1}\in\Real$. Since we can write $e^{A_{\nu+1}^{*}t}=e^{\mu_{\nu+1}t}e^{\Lambda_{\nu+1}^{*}t}$ inequality~\eqref{eqn:whynot}
implies $B_{:\sigma}^{[\nu+1]*}[I_{q}\otimes e^{\Lambda_{\nu+1}^{*}t}]f_{\nu+1}
\leq0$ for all $t\geq \tau$ and $\sigma\in\I_{\nu+1}$. This means the vector
$\tilde f_{\nu+1}=[I_{q}\otimes e^{\Lambda_{\nu+1}^{*}\tau}]f_{\nu+1}$ belongs
to the convex cone $\N_{\nu+1}=\{\zeta:B_{:\sigma}^{[\nu+1]*}[I_{q}\otimes e^{\Lambda_{\nu+1}^{*}t}]\zeta
\leq0\ \mbox{for all}\ \sigma\in\I_{\nu+1}\ \mbox{and}\ t\geq 0\}$. Let $\M_{\nu+1}=({\rm cone}\,[{\rm inc}\,(Q_{:\sigma}^{[\nu+1]})_{\sigma\in\I_{\nu+1}}])^{\circ}$. That is,
$\M_{\nu+1}=\{\zeta:B_{:\sigma}^{[\nu+1]*}[I_{q}\otimes \Lambda_{\nu+1}^{r*}]\zeta\leq0\ \mbox{for all}\ \sigma\in\I_{\nu+1}\ \mbox{and}\ r\}$.
Since $\Q_{\nu+1}$ is the largest subspace contained in ${\rm cone}\,[{\rm inc}\,(Q_{:\sigma}^{[\nu+1]})_{\sigma\in\I_{\nu+1}}]$, the smallest subspace containing
its polar $\M_{\nu+1}$ is $\Q_{\nu+1}^{\perp}$. Since the matrix $[I_{q}\otimes \Lambda_{\nu+1}^{*}]$ is nilpotent, we can invoke Lemma~\ref{lem:pivotal}, which tells us that the subspace $\Q_{\nu+1}^{\perp}$ contains also $\N_{\nu+1}$. Therefore we have
$\tilde f_{\nu+1}\in\Q_{\nu+1}^{\perp}$. Then (since by definition
$[I_{q}\otimes \Lambda_{\nu+1}^{r}]B_{:\sigma}^{[\nu+1]}\in\Q_{\nu+1}$ for all $\sigma\in
\I_{\nu+1}\setminus\I_{\nu+1}^{-}$ and $r$) the structure
(due to $\Lambda_{\kappa}^{n_{\kappa}}=0$)
\begin{eqnarray}\label{eqn:matrixexp}
[I_{q}\otimes e^{\Lambda_{\kappa}^*t}]=\sum_{r=0}^{n_{\kappa}-1}[I_{q}\otimes \Lambda_{\kappa}^{r*}]\frac{t^{r}}{r!}
\end{eqnarray}
allows us to write $B_{:\sigma}^{[\nu+1]*}[I_{q}\otimes e^{\Lambda_{\nu+1}^{*}t}]\tilde f_{\nu+1}=0$ for all $\sigma\in
\I_{\nu+1}\setminus\I_{\nu+1}^{-}$ and $t\geq 0$. Note that $e^{A_{\nu+1}^{*}t}=e^{\mu_{\nu+1}t}e^{\Lambda_{\nu+1}^{*}t}$ and
$\tilde f_{\nu+1}=[I_{q}\otimes e^{\Lambda_{\nu+1}^{*}\tau}]f_{\nu+1}$. Hence we have \eqref{eqn:identicallyzero} for $\nu+1$ for all $t\geq\tau$ thanks to \eqref{eqn:because}. What now remains is that we show \eqref{eqn:identicallyzero} for $\nu=1$. This however
follows from the previous arguments once we note that by \eqref{eqn:incidentallyzero} we have ${\rm Re}\,(B_{:\sigma}^{[1]*}[I_{q}\otimes e^{A_{\kappa}^{*}t}]f_{1})\leq 0$ and that $\I_{1}$ equals $\{1,\,2,\,\ldots,\,p\}$ by definition.

Having established \eqref{eqn:identicallyzero}, let us now be given any index $\gamma\in\{1,\,2,\,\ldots,\,m\}$. If $\gamma=1$ then by letting $\nu=1$ in \eqref{eqn:incidentallyzero} we have ${\rm Re}\,(B_{:\sigma}^{[\gamma]*}[I_{q}\otimes e^{A_{\gamma}^{*}t}]f_{\gamma})\leq 0$ for all
$\sigma\in\I_{\gamma}$ and $t\geq\tau$. Consider now the case $\gamma\geq 2$.
Invoking \eqref{eqn:identicallyzero} with $\nu=\gamma-1$ we obtain $B_{:\sigma}^{[\kappa]*}[I_{q}\otimes e^{A_{\kappa}^{*}t}]f_{\kappa}=0$ for all $\kappa\leq\gamma-1$ and $\sigma\in\I_{\gamma-1}\setminus \I_{\gamma-1}^{-}=\I_{\gamma}$. Then by \eqref{eqn:incidentallyzero} we can write
\begin{eqnarray*}
{\rm Re}\,(B_{:\sigma}^{[\gamma]*}[I_{q}\otimes e^{A_{\gamma}^{*}t}]f_{\gamma})
&=&{\rm Re}\,(B_{:\sigma}^{[\gamma]*}[I_{q}\otimes e^{A_{\gamma}^{*}t}]f_{\gamma})+\sum_{\kappa=1}^{\gamma-1}{\rm Re}\,(B_{:\sigma}^{[\kappa]*}[I_{q}\otimes e^{A_{\kappa}^{*}t}]f_{\kappa})\\
&=&\sum_{\kappa=1}^{\gamma}{\rm Re}\,(B_{:\sigma}^{[\kappa]*}[I_{q}\otimes e^{A_{\kappa}^{*}t}]f_{\kappa})\\
&\leq&0
\end{eqnarray*}
for all $\sigma\in\I_{\gamma}$ and $t\geq\tau$. Hence the result.
\end{proof}

\begin{assumption}\label{assume:eigen}
If $\mu_{\kappa}\notin\Real$ satisfies ${\rm Re}\,\mu_{\kappa}=\mu_{\nu}$ for some
$\nu\in\{1,\,2,\,\ldots,\,m\}$ then $\Lambda_{\kappa}=0$.
\end{assumption}

\begin{theorem}\label{thm:mothership}
Under Assumption~\ref{assume:eigen}, the following two conditions are equivalent.
\begin{enumerate}
\item $\R^{\circ}\subset{\rm null}\,[(e_{k}-e_{\ell})^{*}\otimes I_{n}]$.
\item The below statements simultaneously hold.
\begin{enumerate}
\item ${\rm inc}\,(Q_{:\sigma}^{[\kappa]})_{\sigma\in\I_{\kappa}}$ is strongly $(k,\,\ell)$-connected for all $\mu_{\kappa}\in\Real$.
\item ${\rm inc}\,(Q_{:\sigma}^{[\kappa]})_{\sigma\in\I_{\kappa}}$ is $(k,\,\ell)$-connected for all $\mu_{\kappa}\notin\Real$.
\end{enumerate}
\end{enumerate}
\end{theorem}

\begin{proof}
{\em 1 $\implies$ 2.} Suppose the condition 2a fails. That is,
${\rm cone}\,[{\rm inc}\,(Q_{:\sigma}^{[\nu]})_{\sigma\in\I_{\nu}}]\not\supset
{\rm range}\,[(e_{k}-e_{\ell})\otimes I_{n_{\nu}}]$ for some $\mu_{\nu}\in\Real$.
Without loss of generality assume that
this $\mu_{\nu}$ is the largest real eigenvalue for which 2a fails. Absence of strong $(k,\,\ell)$-connectivity implies that the polar $({\rm cone}\,[{\rm inc}\,(Q_{:\sigma}^{[\nu]})_{\sigma\in\I_{\nu}}])^{\circ}$ is
not contained in
${\rm null}\,[(e_{k}-e_{\ell})^{*}\otimes I_{n_{\nu}}]$. Hence we can find a vector
$f_{\nu}\not\in{\rm null}\,[(e_{k}-e_{\ell})^{*}\otimes I_{n_{\nu}}]$ satisfying $B_{:\sigma}^{[\nu]*}[I_{q}\otimes \Lambda_{\nu}^{r*}]f_{\nu}\leq0$ for all $\sigma\in\I_{\nu}$ and $r\in\{0,\,1,\,\ldots,\,n_{\nu}-1\}$.

For $\kappa<\nu$ and $\mu_{\kappa}\in\Real$ let $\M_{\kappa}=({\rm cone}\,[{\rm inc}\,(Q_{:\sigma}^{[\kappa]})_{\sigma\in\I_{\kappa}}])^{\circ}$. That is,
$\M_{\kappa}=\{f_{\kappa}\in(\Real^{n_{\kappa}})^{q}:B_{:\sigma}^{[\kappa]*}[I_{q}\otimes \Lambda_{\kappa}^{r*}]f_{\kappa}\leq0\ \mbox{for all}\ \sigma\in\I_{\kappa}\ \mbox{and}\ r\}$.
Since $\Q_{\kappa}$ is the largest
subspace contained in ${\rm cone}\,[{\rm inc}\,(Q_{:\sigma}^{[\kappa]})_{\sigma\in\I_{\kappa}}]$, the smallest subspace containing
its polar $\M_{\kappa}$ is $\Q_{\kappa}^{\perp}$. Also, strong $(k,\,\ell)$-connectivity of ${\rm inc}\,(Q_{:\sigma}^{[\kappa]})_{\sigma\in\I_{\kappa}}$ implies $\Q_{\kappa}\supset
{\rm range}\,[(e_{k}-e_{\ell})\otimes I_{n_{\kappa}}]$. Consequently,  $\Q_{\kappa}^{\perp}\subset
{\rm null}\,[(e_{k}-e_{\ell})^{*}\otimes I_{n_{\kappa}}]$. Let $\setP_{\kappa}=\I_{\kappa}\times\{0,\,1,\,\ldots,\,n_{\kappa}-1\}$ and define $\setP_{\kappa}^{-}=\{(\sigma,\,r)\in\setP_{\kappa}:\M_{\kappa}
\cap\{f_{\kappa}:B_{:\sigma}^{[\kappa]*}[I_{q}\otimes \Lambda_{\kappa}^{r*}]f_{\kappa}< 0\}\neq\emptyset\}$. Note that $(\sigma,\,r)\in\setP_{\kappa}^{-}$
means there exist $f_{\kappa}\in\M_{\kappa}\subset\Q_{\kappa}^{\perp}$ satisfying $B_{:\sigma}^{[\kappa]*}[I_{q}\otimes \Lambda_{\kappa}^{r*}]f_{\kappa}\neq 0$ yielding
$[I_{q}\otimes\Lambda_{\kappa}^{r}]B_{:\sigma}^{[\kappa]}
\notin\Q_{\kappa}$. Therefore $(\sigma,\,r)\in\setP_{\kappa}^{-}$ implies $\sigma\in\I_{\kappa}^{-}$. Furthermore, we can write
$[I_{q}\otimes\Lambda_{\kappa}^{r*}]=[I_{q}\otimes\Lambda_{\kappa}^{*}]^{r}$
and it is easy to see that the nilpotency of $\Lambda_{\kappa}$ is inherited by
the matrix $[I_{q}\otimes\Lambda_{\kappa}^{*}]$. Hence
\begin{eqnarray}\label{eqn:monotone}
(\sigma,\,r)\in\setP_{\kappa}^{-} \implies (\sigma,\,r-1)\in\setP_{\kappa}^{-}
\end{eqnarray}
for $r\neq 0$ (see the proof of Lemma~\ref{lem:pivotal}). For each $\kappa<\nu$ with  $\mu_{\kappa}\in\Real$ choose now $f_{\kappa}\in\M_{\kappa}$
that satisfies $B_{:\sigma}^{[\kappa]*}[I_{q}\otimes \Lambda_{\kappa}^{r*}]f_{\kappa}<0$
for all $(\sigma,\,r)\in\setP_{\kappa}^{-}$. Such choice exists thanks to convexity of $\M_{\kappa}$ (see the proof of Lemma~\ref{lem:pivotal}). Then define $y_{\sigma}^{[\kappa]}(t)=B_{:\sigma}^{[\kappa]*}[I_{q}\otimes e^{\Lambda_{\kappa}^*t}]f_{\kappa}$. The implication~\eqref{eqn:monotone} together with
the identity~\eqref{eqn:matrixexp} 
ensure the existence of a scalar $\delta_{\kappa}>0$ and a polynomial $\Delta_{\kappa}(t)$
of order $n_{\kappa}-1$
satisfying
\begin{eqnarray}\label{eqn:ysigma}
y_{\sigma}^{[\kappa]}(t)\leq\left\{
\begin{array}{cl}
-\delta_{\kappa} &\mbox{for}\ \sigma\in\I_{\kappa}^{-}\\
0 &\mbox{for}\ \sigma\in\I_{\kappa}\setminus\I_{\kappa}^{-}\\
\Delta_{\kappa}(t) &\mbox{for}\ \sigma\notin\I_{\kappa}
\end{array}
\right.
\end{eqnarray}
for all $t\geq 0$. As for $y_{\sigma}^{[\nu]}(t)=B_{:\sigma}^{[\nu]*}[I_{q}\otimes e^{\Lambda_{\nu}^*t}]f_{\nu}$ note that we can write
\begin{eqnarray*}
y_{\sigma}^{[\nu]}(t)\leq\left\{
\begin{array}{cl}
0 &\mbox{for}\ \sigma\in\I_{\nu}\\
\Delta_{\nu}(t) &\mbox{for}\ \sigma\notin\I_{\nu}
\end{array}
\right.
\end{eqnarray*}
for some polynomial $\Delta_{\nu}(t)$ of order $n_{\nu}-1$ .
Let us now construct the vector
\begin{eqnarray*}
\eta=\sum_{\kappa\leq\nu,\,\mu_{\kappa}\in\Real}c_{\kappa}[I_{q}\otimes U_{\kappa}]f_{\kappa}
\end{eqnarray*}
where $c_{\nu}=1$ and the remaining scalars
$c_{\kappa}>0$ satisfy
\begin{eqnarray}\label{eqn:ball}
c_{\kappa}\delta_{\kappa}e^{\mu_{\kappa}t}\geq\sum_{\kappa<\gamma\leq\nu,\,\mu_{\gamma}\in\Real} c_{\gamma}\Delta_{\gamma}(t)e^{\mu_{\gamma}t}\,.
\end{eqnarray}
We can always find such scalars because (for real eigenvalues) we have $\mu_{\kappa}>\mu_{\gamma}$ when $\kappa<\gamma$.

Note that for $\kappa<\nu$ we have $f_{\kappa}\in\M_{\kappa}\subset\Q_{\kappa}^{\perp}\subset{\rm null}\,[(e_{k}-e_{\ell})^{*}\otimes I_{n_{\kappa}}]$ which allows us to write
$[(e_{k}-e_{\ell})^{*}\otimes I_{n}][I_{q}\otimes U_{\kappa}]f_{\kappa}= U_{\kappa}[(e_{k}-e_{\ell})^{*}\otimes I_{n_{\kappa}}]f_{\kappa}=0$. Hence
\begin{eqnarray*}
[(e_{k}-e_{\ell})^{*}\otimes I_{n}]\eta
&=&[(e_{k}-e_{\ell})^{*}\otimes I_{n}]\sum_{\kappa\leq\nu,\,\mu_{\kappa}\in\Real} c_{\kappa}[I_{q}\otimes U_{\kappa}]f_{\kappa}\\
&=&\sum_{\kappa\leq\nu,\,\mu_{\kappa}\in\Real} c_{\kappa}U_{\kappa}[(e_{k}-e_{\ell})^{*}\otimes I_{n_{\kappa}}]f_{\kappa}\\
&=& U_{\nu}[(e_{k}-e_{\ell})^{*}\otimes I_{n_{\nu}}]f_{\nu}\\
&\neq&0
\end{eqnarray*}
because $U_{\nu}$ is full column rank and $[(e_{k}-e_{\ell})^{*}\otimes I_{n_{\nu}}]f_{\nu}\neq 0$. Hence $\eta\notin{\rm null}\,[(e_{k}-e_{\ell})^{*}\otimes I_{n}]$. Let us now study the behavior of the entries of $\Bbig^{*}e^{\Abig^{*}t}\eta$. Recall that $B_{:\sigma}^{*}$ denotes the $\sigma$th row of $\Bbig^{*}$. We can write for $\sigma\in\{1,\,2,\,\ldots,\,p\}$
\begin{eqnarray}\label{eqn:decompose}
B_{:\sigma}^{*}e^{\Abig^{*}t}\eta
&=&\sum c_{\kappa}B_{:\sigma}^{*}[I_{q}\otimes e^{A^{*}t}][I_{q}\otimes U_{\kappa}]f_{\kappa}\nonumber\\
&=&\sum c_{\kappa}B_{:\sigma}^{*}[I_{q}\otimes U_{\kappa}][I_{q}\otimes e^{A_{\kappa}^{*}t}]f_{\kappa}\nonumber\\
&=&\sum c_{\kappa}B_{:\sigma}^{[\kappa]*}[I_{q}\otimes e^{A_{\kappa}^{*}t}]f_{\kappa}\nonumber\\
&=&\sum c_{\kappa}B_{:\sigma}^{[\kappa]*}[I_{q}\otimes e^{\mu_{\kappa}t}e^{\Lambda_{\kappa}^{*}t}]f_{\kappa}\nonumber\\
&=&\sum c_{\kappa}e^{\mu_{\kappa}t}y_{\sigma}^{[\kappa]}(t)
\end{eqnarray}
where the summation is through the indices $\kappa\leq\nu,\,\mu_{\kappa}\in\Real$ and
we used the identities $e^{A^{*}t}U_{\kappa}=U_{\kappa}e^{A^{*}_{\kappa}t}$ (because $A^{*}U_{\kappa}=U_{\kappa}A^{*}_{\kappa}$) and $e^{A^{*}_{\kappa}t}=e^{\mu_{\kappa}t}e^{\Lambda^{*}_{\kappa}t}$ (because $A^{*}_{\kappa}=\mu_{\kappa}I_{n_{\kappa}}+\Lambda_{\kappa}^{*}$). Note that for each $\sigma\notin\I_{\nu}$
there is a unique $\kappa_{\sigma}<\nu$ satisfying $\sigma\in\I_{\kappa_{\sigma}}^{-}$ and $\sigma\in\I_{\kappa}$
for $\kappa<\kappa_{\sigma}$. Hence for $\sigma\notin\I_{\nu}$ we can decompose \eqref{eqn:decompose} as
\begin{eqnarray*}
B_{:\sigma}^{*}e^{\Abig^{*}t}\eta
&=&\underbrace{\sum_{\kappa<\kappa_{\sigma}} c_{\kappa}e^{\mu_{\kappa}t}y_{\sigma}^{[\kappa]}(t)}_{\leq 0}
+c_{\kappa_{\sigma}}e^{\mu_{\kappa_{\sigma}}t}y_{\sigma}^{[\kappa_{\sigma}]}(t)
+\sum_{\kappa>\kappa_{\sigma}} c_{\kappa}e^{\mu_{\kappa}t}y_{\sigma}^{[\kappa]}(t)\\
&\leq&-c_{\kappa_{\sigma}}\delta_{\kappa_{\sigma}}e^{\mu_{\kappa_{\sigma}t}}
+\sum_{\kappa>\kappa_{\sigma}}c_{\kappa}\Delta_{\kappa}(t)e^{\mu_{\kappa}t}\\
&\leq&0
\end{eqnarray*}
where we used \eqref{eqn:ball}. If $\sigma\in\I_{\nu}$ then $\sigma\in\I_{\kappa}$
for all $\kappa\leq\nu$ and we have $y_{\sigma}^{[\kappa]}(t)\leq 0$ for all $\kappa\leq\nu$ yielding $B_{:\sigma}^{*}e^{\Abig^{*}t}\eta\leq 0$. Hence
$\Bbig^{*}e^{\Abig^{*}t}\eta\leq 0$ for all $t\geq 0$, i.e., $\eta\in\R^{\circ}$. This implies however $\R^{\circ}\not\subset{\rm null}\,[(e_{k}-e_{\ell})^{*}\otimes I_{n}]$ for we earlier obtained $\eta\notin{\rm null}\,[(e_{k}-e_{\ell})^{*}\otimes I_{n}]$.

Suppose now the condition 2b fails for some $\mu_{\nu}\notin\Real$. Without loss of generality assume that the condition 2a is satisfied and no eigenvalue with real part strictly larger
than ${\rm Re}\,\mu_{\nu}$ violates 2b. That ${\rm inc}\,(Q_{:\sigma}^{[\nu]})_{\sigma\in\I_{\nu}}$
is not $(k,\,\ell)$-connected implies the
existence of a (complex) vector $f_{\nu}\notin{\rm null}\,[(e_{k}-e_{\ell})^{*}\otimes I_{n_{\nu}}]$ satisfying $B_{:\sigma}^{[\nu]*}[I_{q}\otimes \Lambda_{\nu}^{r*}]f_{\nu}=0$ for all $\sigma\in\I_{\nu}$ and $r\in\{0,\,1,\,\ldots,\,n_{\nu}-1\}$. Without loss of generality we can 
assume ${\rm Re}\,(U_{\nu}[(e_{k}-e_{\ell})^{*}\otimes I_{n_{\nu}}]f_{\nu})\neq 0$, for $U_{\nu}$ is full column rank and we can always replace $f_{\nu}$ with $\sqrt{-1}f_{\nu}$. Let $\M_{\kappa}$ and $\setP_{\kappa}^{-}$ be as before. For each $\kappa<\nu$ with $\mu_{\kappa}\in\Real$ choose $f_{\kappa}\in\M_{\kappa}$
that satisfies $B_{:\sigma}^{[\kappa]*}[I_{q}\otimes \Lambda_{\kappa}^{r*}]f_{\kappa}<0$
for all $(\sigma,\,r)\in\setP_{\kappa}^{-}$. Recall $y_{\sigma}^{[\kappa]}(t)=B_{:\sigma}^{[\kappa]*}[I_{q}\otimes e^{\Lambda_{\kappa}^*t}]f_{\kappa}$.
Now \eqref{eqn:monotone} and \eqref{eqn:matrixexp}
allows us to find scalars $\delta_{\kappa}>0$ and polynomials $\Delta_{\kappa}(t)$
of order $n_{\kappa}-1$
satisfying \eqref{eqn:ysigma} for all $t\geq 0$ and all $\kappa<\nu$
with $\mu_{\kappa}\in\Real$. As for the index $\nu$ note that $y_{\sigma}^{[\nu]}(t)$
is complex and we can write
\begin{eqnarray*}
|y_{\sigma}^{[\nu]}(t)|\leq\left\{
\begin{array}{cl}
0 &\mbox{for}\ \sigma\in\I_{\nu}\\
\Delta_{\nu}(t) &\mbox{for}\ \sigma\notin\I_{\nu}
\end{array}
\right.
\end{eqnarray*}
for some polynomial $\Delta_{\nu}(t)$ of order (at most) $n_{\nu}-1$.
If $\Lambda_{\nu}=0$ we let $\Delta_{\nu}(t)$ to be of order zero, i.e., $\Delta_{\nu}(t)\equiv\Delta_{\nu}(0)\geq0$. (This we can do thanks to \eqref{eqn:matrixexp}.)
Let us now construct the vector
\begin{eqnarray*}
\eta={\rm Re}\,([I_{q}\otimes U_{\nu}]f_{\nu})+\sum_{\kappa<\nu,\,\mu_{\kappa}\in\Real}c_{\kappa}[I_{q}\otimes U_{\kappa}]f_{\kappa}
\end{eqnarray*}
where the real scalars $c_{\kappa}>0$ satisfy
\begin{eqnarray}\label{eqn:ballbearing}
c_{\kappa}\delta_{\kappa}e^{\mu_{\kappa}t}\geq \Delta_{\nu}(t)|e^{\mu_{\nu}t}|+\sum_{\kappa<\gamma<\nu,\,\mu_{\gamma}\in\Real} c_{\gamma}\Delta_{\gamma}(t)e^{\mu_{\gamma}t}\,.
\end{eqnarray}
We can always find such scalars thanks to two reasons. First, for $\mu_{\kappa},\,\mu_{\gamma}\in\Real$ we have $\mu_{\kappa}>\mu_{\gamma}$ when $\kappa<\gamma$. Second, if there is $\mu_{\kappa}\in\Real$
satisfying ${\rm Re}\,\mu_{\nu}=\mu_{\kappa}$ then by Assumption~\ref{assume:eigen}
we have $\Delta_{\nu}(t)\equiv\Delta_{\nu}(0)$.

As before for $\kappa<\nu$ we have
$[(e_{k}-e_{\ell})^{*}\otimes I_{n}][I_{q}\otimes U_{\kappa}]f_{\kappa}=0$. Hence
\begin{eqnarray*}
[(e_{k}-e_{\ell})^{*}\otimes I_{n}]\eta
&=& [(e_{k}-e_{\ell})^{*}\otimes I_{n}]{\rm Re}\,([I_{q}\otimes U_{\nu}]f_{\nu})\\
&=& {\rm Re}\,([(e_{k}-e_{\ell})^{*}\otimes I_{n}][I_{q}\otimes U_{\nu}]f_{\nu})\\
&=& {\rm Re}\,(U_{\nu}[(e_{k}-e_{\ell})^{*}\otimes I_{n_{\nu}}]f_{\nu})\\
&\neq&0\,.
\end{eqnarray*}
Hence $\eta\notin{\rm null}\,[(e_{k}-e_{\ell})^{*}\otimes I_{n}]$. Let us now study the behavior of the entries of $\Bbig^{*}e^{\Abig^{*}t}\eta$. We can write for $\sigma\in\{1,\,2,\,\ldots,\,p\}$
\begin{eqnarray}\label{eqn:decompose2}
B_{:\sigma}^{*}e^{\Abig^{*}t}\eta
&=&B_{:\sigma}^{*}[I_{q}\otimes e^{A^{*}t}]{\rm Re}\,([I_{q}\otimes U_{\nu}]f_{\nu})+\sum c_{\kappa}B_{:\sigma}^{*}[I_{q}\otimes e^{A^{*}t}][I_{q}\otimes U_{\kappa}]f_{\kappa}\nonumber\\
&=&{\rm Re}\,(B_{:\sigma}^{[\nu]*}[I_{q}\otimes e^{A_{\nu}^{*}t}]f_{\nu})+\sum c_{\kappa}B_{:\sigma}^{[\kappa]*}[I_{q}\otimes e^{A_{\kappa}^{*}t}]f_{\kappa}\nonumber\\
&=&{\rm Re}\,(B_{:\sigma}^{[\nu]*}[I_{q}\otimes e^{\mu_{\nu}t}e^{\Lambda_{\nu}^{*}t}]f_{\nu})+\sum c_{\kappa}B_{:\sigma}^{[\kappa]*}[I_{q}\otimes e^{\mu_{\kappa}t}e^{\Lambda_{\kappa}^{*}t}]f_{\kappa}\nonumber\\
&=&{\rm Re}\,(e^{\mu_{\nu}t}y_{\sigma}^{[\nu]}(t))+\sum c_{\kappa}e^{\mu_{\kappa}t}y_{\sigma}^{[\kappa]}(t)
\end{eqnarray}
where the summation is through the indices $\kappa<\nu,\,\mu_{\kappa}\in\Real$.
For $\sigma\notin\I_{\nu}$ we can decompose \eqref{eqn:decompose2} as
\begin{eqnarray*}
B_{:\sigma}^{*}e^{\Abig^{*}t}\eta
&=&\underbrace{\sum_{\kappa<\kappa_{\sigma}} c_{\kappa}e^{\mu_{\kappa}t}y_{\sigma}^{[\kappa]}(t)}_{\leq 0}
+c_{\kappa_{\sigma}}e^{\mu_{\kappa_{\sigma}}t}y_{\sigma}^{[\kappa_{\sigma}]}(t)
+\sum_{\kappa_{\sigma}<\kappa<\nu} c_{\kappa}e^{\mu_{\kappa}t}y_{\sigma}^{[\kappa]}(t)
+{\rm Re}\,(e^{\mu_{\nu}t}y_{\sigma}^{[\nu]}(t))\\
&\leq&-c_{\kappa_{\sigma}}\delta_{\kappa_{\sigma}}e^{\mu_{\kappa_{\sigma}t}}
+\sum_{\kappa_{\sigma}<\kappa<\nu} c_{\kappa}\Delta_{\kappa}(t)e^{\mu_{\kappa}t}+\Delta_{\nu}(t)|e^{\mu_{\nu}t}|
\\
&\leq&0
\end{eqnarray*}
where ($\kappa_{\sigma}$ is defined earlier and) we used \eqref{eqn:ballbearing}. If $\sigma\in\I_{\nu}$ then $\sigma\in\I_{\kappa}$
for all $\kappa\leq\nu$ and we have $y_{\sigma}^{[\kappa]}(t)\leq 0$ for all $\kappa<\nu$ as well as $y_{\sigma}^{[\nu]}(t)\equiv 0$. This yields $B_{:\sigma}^{*}e^{\Abig^{*}t}\eta\leq 0$. Hence
$\Bbig^{*}e^{\Abig^{*}t}\eta\leq 0$ for all $t\geq 0$, i.e., $\eta\in\R^{\circ}$. This implies however $\R^{\circ}\not\subset{\rm null}\,[(e_{k}-e_{\ell})^{*}\otimes I_{n}]$ for we already obtained $\eta\notin{\rm null}\,[(e_{k}-e_{\ell})^{*}\otimes I_{n}]$.

{\em 2 $\implies$ 1.} Suppose the condition 1 fails. Then we can find a vector $\eta\in\R^{\circ}$
that satisfies $[(e_{k}-e_{\ell})^{*}\otimes I_{n}]\eta\neq 0$. Let vectors $f_{\kappa}\in(\Complex^{n_{\kappa}})^{q}$ be such that $\eta=\sum_{\kappa=1}^{m}[I_{q}\otimes U_{\kappa}]f_{\kappa}$. We can write
\begin{eqnarray*}
0
\neq[(e_{k}-e_{\ell})^{*}\otimes I_{n}]\eta
=[(e_{k}-e_{\ell})^{*}\otimes I_{n}]\sum_{\kappa=1}^{m}[I_{q}\otimes U_{\kappa}]f_{\kappa}
=\sum_{\kappa=1}^{m} U_{\kappa}[(e_{k}-e_{\ell})^{*}\otimes I_{n_{\kappa}}]f_{\kappa}
\end{eqnarray*}
meaning for some $\kappa$ we have to have $U_{\kappa}[(e_{k}-e_{\ell})^{*}\otimes I_{n_{\kappa}}]f_{\kappa}\neq 0$. This implies
$[(e_{k}-e_{\ell})^{*}\otimes I_{n_{\kappa}}]f_{\kappa}\neq 0$
because $U_{\kappa}$ is full column rank.

Suppose now $\mu_{\kappa}\in\Real$. By Lemma~\ref{lem:ur} there exists $\tau\geq 0$ such that $B_{:\sigma}^{[\kappa]*}[I_{q}\otimes e^{A_{\kappa}^{*}t}]f_{\kappa}\leq 0$ for all
$\sigma\in\I_{\kappa}$ and $t\geq \tau$. Define $\tilde f_{\kappa}=[I_{q}\otimes e^{A_{\kappa}^{*}\tau}]f_{\kappa}$. We can write
\begin{eqnarray}\label{eqn:sour}
[(e_{k}-e_{\ell})^{*}\otimes I_{n_{\kappa}}]\tilde f_{\kappa}
&=& [(e_{k}-e_{\ell})^{*}\otimes I_{n_{\kappa}}][I_{q}\otimes e^{A_{\kappa}^{*}\tau}]f_{\kappa}\nonumber\\
&=& e^{A_{\kappa}^{*}\tau}[(e_{k}-e_{\ell})^{*}\otimes I_{n_{\kappa}}]f_{\kappa}\nonumber\\
&\neq& 0
\end{eqnarray}
 because $[(e_{k}-e_{\ell})^{*}\otimes I_{n_{\kappa}}]f_{\kappa}\neq 0$ and $e^{A_{\kappa}^{*}\tau}$ is nonsingular. Since we can write
$e^{A_{\kappa}^{*}t}=e^{\mu_{\kappa}t}e^{\Lambda_{\kappa}^{*}t}$ and $e^{\mu_{\kappa}t}$
is always positive, $\tilde f_{\kappa}$ belongs to
the cone $\N_{\kappa}=\{\zeta:B_{:\sigma}^{[\kappa]*}[I_{q}\otimes e^{\Lambda_{\kappa}^{*}t}]\zeta\leq 0
\ \mbox{for all}\ \sigma\in\I_{\kappa}\ \mbox{and}\ t\geq 0\}$. Let $\M_{\kappa}=({\rm cone}\,[{\rm inc}\,(Q_{:\sigma}^{[\kappa]})_{\sigma\in\I_{\kappa}}])^{\circ}$. That is,
$\M_{\kappa}=\{\zeta:B_{:\sigma}^{[\kappa]*}[I_{q}\otimes \Lambda_{\kappa}^{r*}]\zeta\leq0\ \mbox{for all}\ \sigma\in\I_{\kappa}\ \mbox{and}\ r\}$.
Since $\Q_{\kappa}$ is the largest
subspace contained in ${\rm cone}\,[{\rm inc}\,(Q_{:\sigma}^{[\kappa]})_{\sigma\in\I_{\kappa}}]$, the smallest subspace containing
its polar $\M_{\kappa}$ is $\Q_{\kappa}^{\perp}$. By Lemma~\ref{lem:pivotal} the subspace $\Q_{\kappa}^{\perp}$
contains also $\N_{\kappa}$. Therefore we have
$\tilde f_{\kappa}\in\Q_{\kappa}^{\perp}$. Then \eqref{eqn:sour} yields $\Q_{\kappa}^{\perp}\not\subset{\rm null}\,[(e_{k}-e_{\ell})^{*}\otimes I_{n_{\kappa}}]$.
Consequently, $\Q_{\kappa}\not\supset{\rm range}\,[(e_{k}-e_{\ell})\otimes I_{n_{\kappa}}]$.
This allows us to write ${\rm cone}\,[{\rm inc}\,(Q_{:\sigma}^{[\kappa]})_{\sigma\in\I_{\kappa}}]\not\supset{\rm range}\,[(e_{k}-e_{\ell})\otimes I_{n_{\kappa}}]$ because ${\rm range}\,[(e_{k}-e_{\ell})\otimes I_{n_{\kappa}}]$ is a subspace and
$\Q_{\kappa}$ is the largest
subspace that ${\rm cone}\,[{\rm inc}\,(Q_{:\sigma}^{[\kappa]})_{\sigma\in\I_{\kappa}}]$ contains. Then $[{\rm inc}\,(Q_{:\sigma}^{[\kappa]})_{\sigma\in\I_{\kappa}}]$ is not strongly $(k,\,\ell)$-connected by definition.

We now consider the other possibility: $\mu_{\kappa}\notin\Real$. By Lemma~\ref{lem:ur} there exists $\tau\geq 0$ such that ${\rm Re}\,(B_{:\sigma}^{[\kappa]*}[I_{q}\otimes e^{A_{\kappa}^{*}t}])f_{\kappa}\leq 0$ for all
$\sigma\in\I_{\kappa}$ and $t\geq \tau$. This implies
$B_{:\sigma}^{[\kappa]*}[I_{q}\otimes e^{A_{\kappa}^{*}t}]f_{\kappa}\equiv 0$ for all
$\sigma\in\I_{\kappa}$ because $\mu_{\kappa}$ is the single distinct eigenvalue of
$A_{\kappa}^{*}$ and it has nonzero imaginary part. (Otherwise, the oscillations would not let
the signal $t\mapsto {\rm Re}\,(B_{:\sigma}^{[\kappa]*}[I_{q}\otimes e^{A_{\kappa}^{*}t}]f_{\kappa})$ stay nonpositive indefinitely.) Then the identity $e^{A_{\kappa}^{*}t}=e^{\mu_{\kappa}t}e^{\Lambda_{\kappa}^{*}t}$ implies
$B_{:\sigma}^{[\kappa]*}[I_{q}\otimes e^{\Lambda_{\kappa}^{*}t}]f_{\kappa}\equiv 0$
for all
$\sigma\in\I_{\kappa}$. By differentiating $B_{:\sigma}^{[\kappa]*}[I_{q}\otimes e^{\Lambda_{\kappa}^{*}t}]f_{\kappa}\equiv 0$ (with respect to $t$) sufficiently many times and evaluating the derivatives at $t=0$ we at once see that $f_{\kappa}$ belongs to ${\rm null}\,[{\rm inc}\,(Q_{:\sigma}^{[\kappa]})_{\sigma\in\I_{\kappa}}]^{*}$. Since $[(e_{k}-e_{\ell})^{*}\otimes I_{n_{\kappa}}]f_{\kappa}\neq 0$, this means ${\rm null}\,[{\rm inc}\,(Q_{:\sigma}^{[\kappa]})_{\sigma\in\I_{\kappa}}]^{*}\not\subset{\rm null}\,[(e_{k}-e_{\ell})^{*}\otimes I_{n_{\kappa}}]$. Hence ${\rm range}\,[{\rm inc}\,(Q_{:\sigma}^{[\kappa]})_{\sigma\in\I_{\kappa}}]\not\supset{\rm range}\,[(e_{k}-e_{\ell})\otimes I_{n_{\kappa}}]$, i.e., $[{\rm inc}\,(Q_{:\sigma}^{[\kappa]})_{\sigma\in\I_{\kappa}}]$ is not $(k,\,\ell)$-connected.
\end{proof}
\vspace{0.12in}

The proof of the below result is similar to that of Lemma~\ref{lem:onade}.

\begin{lemma}\label{lem:onaid}
The following are equivalent.
\begin{enumerate}
\item The array $[A,\,(B_{::})]$ is positively $(k,\,\ell)$-controllable.
\item $\R\supset{\rm range}\,[(e_{k}-e_{\ell})\otimes I_{n}]$.
\end{enumerate}
\end{lemma}

\begin{assumption}\label{assume:closed}
If ${\rm cl}\,\R\supset{\rm range}\,[(e_{k}-e_{\ell})\otimes I_{n}]$ then
$\R\supset{\rm range}\,[(e_{k}-e_{\ell})\otimes I_{n}]$.
\end{assumption}

At the time of writing this paper we do not know whether it is possible that Assumption~\ref{assume:closed} is violated. For the benchmark example of network of double integrators it is not difficult to see that it holds. In general, Assumption~\ref{assume:closed} can be shown to hold for any array $[A,\,(B_{::})]$ of chain of integrators with
\begin{eqnarray*}
A = \left[\begin{array}{ccccc}
0&1&0&\cdots &0\\
0&0&1&\cdots &0\\
\vdots&\vdots&\vdots&\ddots&\vdots\\
0&0&0&\cdots&1\\
0&0&0&\cdots&0
\end{array}\right]\quad\mbox{and}\quad{\rm inc}\,(B_{::})=G\otimes\left[\begin{array}{c}
0\\
0\\
\vdots\\
0\\
1
\end{array}\right]
\end{eqnarray*}
where $G\in\Real^{q\times p}$ is an incidence matrix with columns of the form $e_{i}-e_{j}$.

\begin{theorem}\label{thm:fathership}
Under Assumptions~\ref{assume:eigen} and~\ref{assume:closed}, the following two conditions are equivalent.
\begin{enumerate}
\item The array $[A,\,(B_{::})]$ is positively $(k,\,\ell)$-controllable.
\item The below statements simultaneously hold.
\begin{enumerate}
\item ${\rm inc}\,(Q_{:\sigma}^{[\kappa]})_{\sigma\in\I_{\kappa}}$ is strongly $(k,\,\ell)$-connected for all $\mu_{\kappa}\in\Real$.
\item ${\rm inc}\,(Q_{:\sigma}^{[\kappa]})_{\sigma\in\I_{\kappa}}$ is $(k,\,\ell)$-connected for all $\mu_{\kappa}\notin\Real$.
\end{enumerate}
\end{enumerate}
\end{theorem}

\begin{proof}
{\em 1$\implies$2.} Suppose the array $[A,\,(B_{::})]$ is positively $(k,\,\ell)$-controllable.
By Lemma~\ref{lem:onaid} we have $\R\supset{\rm range}\,[(e_{k}-e_{\ell})\otimes I_{n}]$.
This implies $\R^{\circ}\subset{\rm null}\,[(e_{k}-e_{\ell})^{*}\otimes I_{n}]$. Then by Assumption~\ref{assume:eigen} and Theorem~\ref{thm:mothership} the second condition follows.

{\em 2$\implies$1.} Suppose the second condition holds. By Assumption~\ref{assume:eigen} and Theorem~\ref{thm:mothership} we have $\R^{\circ}\subset{\rm null}\,[(e_{k}-e_{\ell})^{*}\otimes I_{n}]$. This implies $\R^{\circ\circ}\supset{\rm range}\,[(e_{k}-e_{\ell})\otimes I_{n}]$.
Since $\R^{\circ\circ}={\rm cl}\,\R$ Assumption~\ref{assume:closed} yields $\R\supset{\rm range}\,[(e_{k}-e_{\ell})\otimes I_{n}]$. Then by Lemma~\ref{lem:onaid} the array $[A,\,(B_{::})]$ is strongly $(k,\,\ell)$-controllable.
\end{proof}

\section{Conclusion}

For networks of relatively actuated LTI systems we established in this paper that certain controllability properties of an array and certain connectivity properties of a set of graphs obtained from the array are equivalent. The main findings rested on four theorems. First, in  Theorem~\ref{thm:con} we presented the equivalence between array controllability and graph connectivity. Then in Theorem~\ref{thm:poscon} we stated that an array can be steered by positive controls if the constructed graphs are strongly connected. Those two theorems in the first half of the paper were related to the overall controllability of the array. In the second half we focused on the problem of controlling the difference of the states of a particular pair of systems in the array. To this end, in Theorem~\ref{thm:klcon} we obtained that this pairwise controllability can be understood through pairwise connectivity of certain graphs. Finally, in Theorem~\ref{thm:fathership} we showed that positive pairwise controllability is closely related to strong pairwise connectivity of the graphs associated to the array.

\bibliographystyle{plain}
\bibliography{references}
\end{document}